\documentclass[12pt]{amsart}

\usepackage{enumerate}

\usepackage[active]{srcltx}
\usepackage{nicefrac}

\usepackage{amsthm,amssymb,setspace}
\usepackage[pagebackref,colorlinks,linkcolor=red,citecolor=blue,urlcolor=blue,hypertexnames=true]{hyperref}
\usepackage{amsrefs}
\usepackage[matrix, arrow]{xy}

\theoremstyle{plain}
\newtheorem{theorem}{Theorem}[section]

\newtheorem{thmstar}{Theorem}

\newtheorem{corollary}[theorem]{Corollary}
\newtheorem{lemma}[theorem]{Lemma}
\newtheorem{proposition}[theorem]{Proposition}

\theoremstyle{definition}
\newtheorem{question}[theorem]{Question}

\theoremstyle{remark}

\newtheorem{remark}[theorem]{Remark}

\allowdisplaybreaks[4]

\begin{document}

\onehalfspace

\title[]{Factors with prescribed number of invariant  subalgebras not arising from subgroups}

\author{Yongle Jiang}
\address{Yongle Jiang, School of Mathematical Sciences, Dalian University of Technology, Dalian, 116024, China}
\email{yonglejiang@dlut.edu.cn}

\author{Qinxuan Xu*}
\address{Qinxuan Xu (Corresponding author), School of Mathematical Sciences, Dalian University of Technology, Dalian, 116024, China}
\email{qxqxq609@163.com}
\thanks{*-corresponding author}

\begin{abstract}
 For any given integer $n\geq 1$, we construct i.c.c. groups $G$ such that the II$_1$ factors $L(G)$ have exactly $n$-many $G$-invariant von Neumann subalgebras not arising from subgroups.
\end{abstract}



\maketitle


\section{Introduction}

The study of invariant von Neumann subalgebras in group von Neumann algebras has received considerable attention in recent years. Motivated by the pioneering works \cites{AB,CD,KP}, Amrutam and the first named author introduced
the invariant von Neumann subalgebras rigidity (ISR) property  \cite{AJ}. A countable discrete group $G$
is said to satisfy the \emph{ISR property} if every 
$G$-invariant von Neumann subalgebra 
$M\subseteq L(G)$ is of the form $L(H)$ for some normal subgroup $H\lhd G$, where $G$ acts on $L(G)$ by conjugation. In other words, for groups with the ISR property, the lattice of $G$-invariant von Neumann subalgebras of 
$L(G)$ is completely determined by the lattice of normal subgroups of 
$G$, hence precluding the existence of any ``exotic" invariant subalgebras that do not arise from underlying group-theoretic structure. It can also be thought of as an ingredient to classify the simplest type, i.e. the Dirac type of invariant random von Neumann subalgebras, a new concept introduced in \cite{AHO}.

Since its introduction, the ISR property has been established for a broad array of groups. Amrutam-Jiang originally proved that many ``negatively curved" groups possess the ISR property \cite{AJ}, including all torsion-free non-amenable hyperbolic groups and torsion-free groups with positive first 
$\ell^2$-Betti number under mild assumptions, as well as certain finite direct products thereof. Subsequent works have extended these results in several directions. For instance,  Chifan-Das-Sun proved that all acylindrically hyperbolic groups with trivial amenable radical have this ISR property \cite{CDS}. The first example of an infinite amenable group with the ISR property was constructed by Jiang-Zhou \cite{JZ}, showing that the finitary permutation group 
$S_{\infty}$ enjoys this rigidity property. More recently, Dudko-Jiang developed a character approach to the ISR property \cite{DJ}, which has been proven to be effective in treating new classes of both amenable and non-amenable groups in later development \cites{ADJS,jiangli}. 
These developments have enriched our understanding of when the ISR property holds, although it is still unclear whether all non-amenable groups with trivial amenable radical have this property or not. We also remark that a relative, and in fact stronger version of this ISR property has been explored quite recently by Amrutam in \cite{A-relative}.

In contrast, far less is known about the structure of 
invariant von Neumann subalgebras in case that the ISR property fails. In such cases,  a complete classification of all 
$G$-invariant von Neumann subalgebras in $L(G)$ becomes a subtle and challenging problem. To date, only a handful of complete classifications have been achieved for groups without the ISR property. In each of these examples, the number of invariant subalgebras not coming from subgroups is either exactly one or infinitely many.

\begin{itemize}
    \item The case of \textbf{exactly one} exotic invariant subalgebra: In \cite{ADJS}, Amrutam-Dudko-Jiang-Skalski proved a classification result (Theorem B therein) where the group 
$G$ fails the ISR property, yet admits precisely one 
$G$-invariant von Neumann subalgebra beyond those arising from normal subgroups. This provides an interesting example where the failure of rigidity is minimal.
\item The case of \textbf{infinitely many} exotic invariant subalgebras: This phenomenon appears in the works of Jiang-Liu \cite{jiangliu} and Jiang-Li \cite{jiangli}, who completely classified all 
$G$-invariant von Neumann subalgebras in 
$L(G)$ for 
$G=\mathbb{Z}^n\rtimes SL(n,\mathbb{Z})$ with 
$n\geq 2$. 
\end{itemize}

These examples raise a natural and intriguing question:

\begin{question}\label{motivating-question}
For any given integer 
$n\geq 2$, can one construct a countable discrete group 
$G$ that does not satisfy the ISR property, such that the $G$-invariant von Neumann subalgebras in 
$L(G)$ can be completely classified, and the number of such subalgebras not arising from subgroups is exactly 
$n$? 
\end{question}

While it is not hard to construct groups $G$ with $2^n$-many such exotic invariant subalgebras for any $n\geq 1$ using known results (see Remark \ref{remark: construct 2^n many exotic invariant subalgebras}), the general case needs a new construction. In fact,
this question can be viewed as an analogue, in the setting of invariant von Neumann subalgebras, of a classical and important problem in the theory of 
II$_1$ factors: 
\begin{question}
Given any integer 
$n\geq 2$, can one construct a 
II$_1$ factor 
$M$ possessing exactly 
$n$ Cartan subalgebras, or exactly 
$n$ group measure space Cartan subalgebras, up to automorphism or unitary conjugacy? 
\end{question}
This question was largely resolved by Krogager-Vaes \cite{KV}, who constructed, for every positive integer 
$n$, a class of II$_1$ factors 
$M$ that admit exactly 
$n$ group measure space Cartan subalgebras up to conjugacy by an automorphism of $M$. For initial results on studying this problem, we refer the reader to the introduction part of \cite{KV} and the references therein.

In the present paper, we establish an analogous result in the context of invariant von Neumann subalgebras, which provides an affirmative answer to Question \ref{motivating-question}.

\begin{thmstar}
 \phantomsection\label{thmA}
\begin{itemize}
\item For every integer $n \geq 1$, there exists an amenable i.c.c. group $G_n$ such that the II$_1$ factor $L(G_n)$ has exactly $n$ $G_n$-invariant von Neumann subalgebras in $L(G_n)$ not arising from subgroups of $G_n$.
\item\label{thmA2} For every integer $n \geq 1$, there exists a non-amenable i.c.c. group $H_n$ such that the II$_1$ factor $L(H_n)$ has exactly $2n$ $H_n$-invariant von Neumann subalgebras in $L(H_n)$ not arising from subgroups of $H_n$.
\end{itemize}
\end{thmstar}

In fact, Theorem \ref{thmA}, proved as Corollary \ref{cor: group with n-many exotic invariant vN subalgebras inside L(G)} and Corollary \ref{cor: non-amenable group with n-many exotic invariant vN subalgebras inside L(G)},  follows from a full classification of invariant subalgebras in $L(G_n)$ and $L(H_n)$, see Theorem \ref{thm: invariant vN subalgebras inside L(k rtimes k*)} and Corollary \ref{cor: invariant vN subalgebras in L(G×H)}.

Our construction draws inspiration from the works of \cites{ADJS,jiangliu,jiangli}. Indeed, the group $G_n$ is also constructed to be certain semi-direct product group $A\rtimes \Gamma$, where $A$ is abelian, so that we can explore the algebraic action $\Gamma\curvearrowright(\widehat{A},\text{Haar})$ that it induces. But the translation of their ideas to attack Question \ref{motivating-question} requires several new technical ingredients from ergodic theory, field theory, character theory and arithmetic observations. For example, we expect all exotic invariant subalgebras arise precisely from non-trivial factors of the above algebraic action, hence we do need a complete  classification of all non-trivial factors such that the number of such factors equals a prescribed number. The starting point is to understand \cite[Remark 1.4]{BF}, see Lemma \ref{lem: 2-simple} for a detailed proof, which provides us with a natural candidate, i.e. $\mathbf{k}\rtimes \mathbf{k}^*$ for a suitable chosen field $\mathbf{k}$, to work with. 

However, even with the above mentioned candidate at hand, new challenges arise while trying to classify invariant subalgebras. In particular, we highlight that there exists a new rigid phenomenon, i.e. Lemma \ref{lem: invariant A with Z(A)=L(k)^{F}} that does not seem to arise in all previous works, see Remark \ref{remark: direct integral gives more intermediate ones} for more discussion on this.
We anticipate that the techniques developed here may find further applications in the study of intermediate subalgebras and rigidity phenomena in operator algebras.

The paper is organized as follows: besides this section, there are three more sections. In Section \ref{section: preliminaries}, we collect facts on algebraic actions, 
torsion subgroups of multiplicative groups over fields, joinings and characters. Then we show the full classification of factors for the algebraic action $\mathbf{k}^*\curvearrowright\widehat{\mathbf{k}}$ in Section \ref{section: factors of algebraic actions}. The proof of Theorem \ref{thmA} will be given in Section \ref{section: invariant subalgebras}, see Corollary \ref{cor: group with n-many exotic invariant vN subalgebras inside L(G)} and Corollary \ref{cor: non-amenable group with n-many exotic invariant vN subalgebras inside L(G)}. In fact, both of them follow from the full classification of invariant subalgebras in $L(\mathbf{k}\rtimes \mathbf{k}^*)$ as in Theorem \ref{thm: invariant vN subalgebras inside L(k rtimes k*)}. 

\subsection*{Acknowledgments}
Both authors are supported by the National Natural Science Foundation of China (Grant No. 12471118). Besides, this work was partially supported by the Simons Foundation grant (award no. SFI-MPS-T-Institutes-00010825) and from State Treasury funds as part of a task commissioned by the Minister of Science and Higher Education under the project ``Organization of the Simons Semesters at the Banach Center - New Energies in 2026-2028'' (agreement no. MNiSW/2025/DAP/491). J.Y. thanks Prof. Manfred Einsiedler for a useful email correspondence dating back to 2020. We are grateful to Prof. Tattwamasi Amrutam and Prof. Hanfeng Li for very helpful comments and pointing out several inaccuracies in a previous version.

\section{Preliminaries}\label{section: preliminaries}

In this section, we collect some basic facts on algebraic actions, finite subgroups of multiplicative groups over fields, joinings and characters. All of these will be applied to the semi-direct product group $\mathbf{k}\rtimes \mathbf{k}^*$ and the induced algebraic action $\mathbf{k}^*\curvearrowright\widehat{\mathbf{k}}$, where $\mathbf{k}$ is a suitable chosen field.

\subsection{Algebraic actions}

Let $\Gamma$ be a countable discrete group.
For a compact metrizable abelian group $X$ equipped with the normalized Haar measure $\mu$, an action $\alpha: \Gamma\curvearrowright (X,\mu)$ is called an \emph{algebraic action} if $\alpha: \Gamma\rightarrow Aut(X)$ is a group homomorphism from $\Gamma$ to the continuous automorphism group $Aut(X)$. Note that the Pontryagin dual $\widehat{X}$ inherits a left $\mathbb{Z}\Gamma$-module structure. Conversely, given a countable left $\mathbb{Z}\Gamma$-module $M$, then it induces an algebraic action $\Gamma\curvearrowright\widehat{M}$ defined by $\langle g\chi,m\rangle:=\langle \chi, g^{-1}m\rangle$ for all $g\in \Gamma$, $\chi\in \widehat{M}$ and all $m\in M$, where the compact abelian group $\widehat{M}$ (the Pontryagin dual of $M$)  is equipped with the Haar measure and $\langle \cdot,\cdot\rangle$ denotes the pairing between $\widehat{M}$ and $M$. Hence, for an countable abelian group $A$, the semi-direct product group $G:=A\rtimes \Gamma$ induces an algebraic action $\Gamma\curvearrowright \widehat{A}$. 

In this paper, we will work with $G=\mathbf{k}\rtimes \mathbf{k}^*$, where $\mathbf{k}$ is a suitable chosen infinite field, $\mathbf{k}^*$ denotes the multiplicative group of non-zero elements in $\mathbf{k}$ and $\mathbf{k}^*$ acts on $\mathbf{k}$ by multiplication. Recall that for an algebraic action $\Gamma\curvearrowright (X,\mu)$, it is  mixing if and only if  the stabilizer subgroup of every non-trivial element in $\widehat{X}$ is finite \cite[Proposition 2.35]{KL_book}. Hence the action $\mathbf{k}^*\curvearrowright \widehat{\mathbf{k}}$ is mixing since every non-trivial element in $\widehat{\mathbf{k}}$ has trivial stabilizer subgroup.

For more information on algebraic actions, we refer the readers to 
\cites{Schmidt_book, KL_book}.

\subsection{A fact on multiplicative groups over fields}

To control the number of non-trivial factors for the algebraic action $\mathbf{k}^*\curvearrowright\widehat{\mathbf{k}}$, we will need to control the number of finite subgroups in $\mathbf{k}^*$. Hence we prepare such a fact in the following lemma.

\begin{lemma}\label{lem: k* which has exactly n-many subgroups}
    Let $n\geq 1$. Then there exists a countable discrete field $\mathbf{k}$ of characteristic zero such that there are precisely $n$-many non-trivial finite subgroups inside $\mathbf{k}^*$.
\end{lemma}
\begin{proof}
We claim that for a suitable choice of $m\in\mathbb{N}$, we have that $\mathbf{k}:=\mathbb{Q}(\zeta_m)$ satisfies the desired property, where $\zeta_m$ denotes the $m$-th primitive root of unity $e^{2\pi i/m}$.

For any $n\geq 1$, we may take $m=2^{n}$.  Since any finite subgroup of $\mathbf{k}^*$ must be a cyclic group consisting of roots of unity \cite[Lemma 6.1]{Mor} and in fact the roots of unity in $\mathbb{Q}(\zeta_m)$ form a cyclic group $\langle \zeta_m\rangle$ of order $m$ for even $m$ (see e.g. \cite[Theorem 4.2]{moor}), we deduce that for $\mathbf{k}=\mathbb{Q}(\zeta_{2^{n}})$, finite subgroups inside $\mathbf{k}^*$ are listed as:
\begin{align*}
    \{1\}, \langle \zeta_2\rangle, \langle \zeta_{2^2}\rangle,\ldots,\langle \zeta_{2^{n}}\rangle.
\end{align*}
Hence it has exactly $n$-many non-trivial finite subgroups.  
\end{proof}
\begin{remark}
 If $n\geq 1$ satisfies that $n+1$ is prime, then there does not exist $m\neq 2^n$ satisfying that for $\mathbf{k}=\mathbb{Q}(\zeta_m)$, $\mathbf{k}^*$ has exactly $n$-many non-trivial finite subgroups. Except for this case, there exists $m\neq 2^n$ such that $\mathbf{k}=\mathbb{Q}(\zeta_m)$ satisfies the desired condition.
\end{remark}
\begin{proof}
    Note that the roots of unity in $\mathbb{Q}(\zeta_m)$ form a cyclic group $\langle \zeta_m\rangle$ of order $m$ for even $m$ and form a cyclic group $\langle \zeta_{2m}\rangle$ of order $2m$ for odd $m$ (see e.g. \cite[Theorem 4.2]{moor}). Thus the number of non-trivial finite subgroups of $\mathbf{k}^*$ is $d(m)-1$ for even $m$ and is $d(2m)-1=2d(m)-1$ for odd $m$, where $d(m)$ denotes the number of divisors of $m$.  
    
    Let $m=p_1^{a_1}\cdots p_r^{a_r}$ be the prime factorization of $m$. Assume that for $\mathbf{k}=\mathbb{Q}(\zeta_m)$, $\mathbf{k}^*$ has exactly $n$-many non-trivial finite subgroups. Then $n+1=d(m)=(a_1+1)\cdots(a_r+1)$ if $m$ is even, and $n+1=2d(m)=2(a_1+1)\cdots(a_r+1)$ if $m$ is odd. Therefore, it is easy to check that $2^n$ is the unique positive integer $m$ such that $\mathbf{k}=\mathbb{Q}(\zeta_m)$ satisfies the desired condition when $n+1$ is prime. Except for this case, we may write $n+1=q_1^{b_1}\cdots q_s^{b_s}$ for its prime factorization, where $s\geq 2$. Thus we may take $m=2^{q_1^{b_1}-1}r_2^{q_2^{b_2}-1}\cdots r_s^{q_s^{b_s}-1}$ to obtain the desired, where $2,r_2,\cdots, r_s$ are distinct prime numbers.
\end{proof}

\subsection{Joining}  
We refer the readers to \cite[Chapter 6]{Gla} for the following definitions and facts.

Let $\alpha_i: G\curvearrowright (X_i,\mu_i)$ ($i=1,2$) be two ergodic p.m.p. actions. Then a probability measure $\mu$ on the product space $X_1\times X_2$ is called a \emph{joining} of $\alpha_1$ and $\alpha_2$ if it is $G$-invariant and $(\pi_i)_*\mu=\mu_i$, where $\pi_i: X_1\times X_2\rightarrow X_i$ denotes the projection onto the $i$-th coordinate. When $\alpha_1=\alpha_2$, we refer to such joinings as \emph{self-joinings}.

Let $\phi: (X_1,\mu_1)\rightarrow (X_2,\mu_2)$ be a $G$-equivariant measure preserving map. Let $gr(\mu_1,\phi)=\int_X(\delta_x\times \delta_{\phi(x)})d\mu_1(x)=(id\times \phi)_*\mu_1$, where $(id\times \phi): X_1\rightarrow X_1\times X_2$ is the map given by $(id\times \phi)(x)=(x,\phi(x))$. Equivalently, $gr(\mu_1,\phi)(A\times B)=\mu_1(A\cap \phi^{-1}B)$. When $\phi$ is an isomorphism of the two ergodic actions, we say that $gr(\mu_1,\phi)$ is an \emph{isomorphism graph joining}. For a weakly mixing p.m.p. action $G\curvearrowright (X,\mu)$ (e.g. $\mathbf{k}^*\curvearrowright \widehat{\mathbf{k}}$), we call it \emph{2-simple} if each ergodic self-joining is either  an isomorphism graph joining of some automorphism  $\phi\in Aut_G(X, \mu)$ or $\mu\times \mu$. Here, $Aut_G(X,\mu)$ is the automorphism group consisting of all $G$-equivariant measure preserving isomorphisms of  $(X,\mu)$. With a Cantor model for $X$, the topology on $Aut_G(X,\mu)$ is given by the metric \[d(\phi, \psi)=\sum_{n=1}^\infty\frac{1}{2^n}\mu(\phi^{-1}A_n\Delta\psi^{-1}A_n),\] where $A_n$ is a list of clopen subsets of $X$. With this topology $Aut_G(X,\mu)$ is a Polish group, and this topology does not depend on the Cantor model chosen for $X$.

A crucial property on 2-simple actions was proved by Veech \cite[Theorem 12.3]{Gla}, showing that every non-trivial factor of a 2-simple p.m.p. action is defined by modding out some compact subgroup of $Aut_G(X,\mu)$.

\subsection{Characters} Let $G$ be a countable discrete group. A map $\phi: G\rightarrow\mathbb{C}$ is called a \emph{character} on $G$ if satisfies the following three conditions:
\begin{itemize}
    \item $\phi$ is \emph{positive definite}, i.e. $\sum_{i=1}^n\sum_{j=1}^n\overline{\alpha_i}\alpha_j\phi(s_i^{-1}s_j)\geq 0$ for all $n\in\mathbb{N}$ and any elements $s_1,\ldots, s_n\in G$ and any $\alpha_1,\ldots,\alpha_n\in\mathbb{C}$.
    \item $\phi$ is \emph{conjugation-invariant}, i.e. $\phi(sts^{-1})=\phi(s)$ for all $s, t\in G$.
    \item $\phi$ is \emph{normalized}, i.e. $\phi(e)=1$.
\end{itemize}
Let $A\subseteq L(G)$ be a $G$-invariant von Neumann subalgebra. Denote by $E: L(G)\rightarrow A$ the canonical trace $\tau$-preserving conditional expectation on $A$.
A basic fact we will use is that the map $\phi: G\rightarrow \mathbb{C}$ defined by $\phi(g):=\tau(v_g^*E(v_g))$ is a character, where $v_g$ denotes the canonical unitary defined by $g$, i.e. $v_g\in \mathcal{U}(\ell^2(G))$ is defined by $v_g(\delta_s)=\delta_{gs}$ for all $s\in G$.
For a proof of this fact, see \cite[Proposition 3.2]{JZ} (Note that characters are instead called traces in \cite{JZ}).  

The other useful fact is the following well-known trick on characters, which has been used in \cites{Bek,DM}. For a proof, see \cite[Lemma 2.7]{DJ}. 

\begin{proposition}\label{prop: vanishing characters}
    Let $G$ be a countable discrete group. Let $\phi$ be a character on $G$. Let $g\in G$. Assume that $\phi(s_n^{-1}s_m)=0$ for all $n\neq m$, where for some $t_n\in G$, we have that $s_n:=t_ngt_n^{-1}$ are infinitely many pairwisely distinct conjugates of $g$. Then $\phi(g)=0$.
\end{proposition}

\section{Factors of the algebraic action $\mathbf{k}^*\curvearrowright\widehat{\mathbf{k}}$}\label{section: factors of algebraic actions}

In this section, following \cite[Remark 1.4]{BF}, we classify all factors of $\mathbf{k}^*\curvearrowright \widehat{\mathbf{k}}$ by showing that it is a 2-simple action and hence we can apply Veech's classical theorem mentioned above. 

Let $\mathbf{k}$ be a countably infinite discrete field (not necessarily of characteristic zero) and let $\mathbf{k}^*$ denote the group of all non-zero elements of $\mathbf{k}$ under multiplication. We denote by $\widehat{\mathbf{k}}$ the Pontryagin dual of $(\mathbf{k},+)$. Let $\mathbf{k}^*\curvearrowright(\widehat{\mathbf{k}},m)$ be the natural algebraic action, where $m$ denotes the normalized Haar measure on $\widehat{\mathbf{k}}$.  

The following fact was recorded as \cite[Remark 1.4]{BF}. We give a  proof for completeness. This serves as the starting point for our construction.

\begin{lemma}\label{lem: 2-simple}
The algebraic action $\mathbf{k}^*\curvearrowright\widehat{\mathbf{k}}$ is 2-simple.   
\end{lemma}
\begin{proof}
    By definition of 2-simple, we need to check that every ergodic joining $\nu$ of the given action on $\widehat{\mathbf{k}}\times\widehat{\mathbf{k}}$ must be either $m\times m$ or a graph joining of some automorphism, i.e. there exists some $\phi\in\mathit{Aut_{\mathbf{k}^*}}(\widehat{\mathbf{k}},m)$ such that $\nu=(id, \phi)_*m$ via the map $\widehat{\mathbf{k}}\ni x\overset{\text{$(id,\phi)$}}{\mapsto} (x,\phi(x))\in \widehat{\mathbf{k}}\times\widehat{\mathbf{k}}$. 
    
    By Lemma 2.4 and Lemma 2.8 in \cite{BF}, we deduce that $\nu=m_{W^{\perp}}$ for some linear subspace $W\subseteq\mathbf{k}\oplus \mathbf{k}=:V$, where $m_{W^{\perp}}$ denotes the normalized Haar measure on ${W^{\perp}}$. Note that $m \times m$ is the normalized Haar measure on $\widehat{\mathbf{k}\times \mathbf{k}}\cong \widehat{\mathbf{k}}\times\widehat{\mathbf{k}}$ and $m_{W^{\perp}}$ is the restriction of $m \times m$ on subgroup ${W^{\perp}}\subseteq \widehat{\mathbf{k}\times \mathbf{k}}$.

    By definition, \[W^{\perp}=\{v\in \widehat{\mathbf{k}}\times\widehat{\mathbf{k}}: v(x)=1, \forall x \in W\}.\] 

    As a vector subspace, $W$ has dimension 0, 1, or 2.

    \textbf{Case 1.} $dim_{\mathbf{k}}W=0$.

    In this case, $W=\{0\oplus 0\}$. Then $W^{\perp}=\widehat{\mathbf{k}}\times\widehat{\mathbf{k}}$ and thus $\nu=m\times m$. 

    \textbf{Case 2.} $dim_{\mathbf{k}}W=1$.

    Pick any $0\neq(c_1, c_2)\in W$. Then $W=\{(cc_1, cc_2): c\in \mathbf{k}\}$. We may discuss two cases.

    \textbf{Subcase 1.} $c_1=0$ or $c_2=0$. Then $W=0\oplus\mathbf{k}$ or $W=\mathbf{k}\oplus0$. 

    If $W=0\oplus\mathbf{k}$, then $W^{\perp}=\widehat{\mathbf{k}}\times\{1\}$. Thus $\nu=m_{\widehat{\mathbf{k}}\times\{1\}}$, which contradicts to the marginal of $\nu$ is $m$. By symmetry, $W=\mathbf{k}\oplus0$ is also impossible.

    \textbf{Subcase 2.} $c_1c_2\neq0$. 

    We may write $W=\{(c,c\lambda):c\in\mathbf{k}\}$, where $\lambda=\frac{c_2}{c_1}\neq 0$.

     Let us determine $W^{\perp}$. By definition, $v=(v_1,v_2)\in\widehat{\mathbf{k}}\times\widehat{\mathbf{k}}$ lies in $W^{\perp}$ if and only if $v(x)=1$ for all $x\in W$. In other words, \[v_1(c)v_2(c\lambda)=1, \forall c\in \mathbf{k}.\] By definition of $\mathbf{k}^*\curvearrowright\widehat{\mathbf{k}}$, $v_2(c\lambda)=(\frac{1}{\lambda}·v_2)(c)$. Thus $W^{\perp}=\{(v_1,v_2)\in\widehat{\mathbf{k}}\times\widehat{\mathbf{k}} : v_1(\frac{1}{\lambda}v_2)=1\}$.

     Define $\phi(v)(t)=v(ts^{-1})$ for all $t\in \mathbf{k}$ and $v\in \widehat{\mathbf{k}}$, where $s=-\lambda$. Clearly, $\phi$ is $\mathbf{k}^*$-equivariant. Moreover, it is easy to check that $\phi$ is a topological group automorphism of $\widehat{\mathbf{k}}$, which implies that $\phi$ preserves the normalized Haar measure $m$ on $\widehat{\mathbf{k}}$. Thus $\phi\in \mathit{Aut_{\mathbf{k}^*}}(\widehat{\mathbf{k}},m)$. We now prove that $\nu=(id,\phi)_*m$. 

     Consider the following group isomorphism \[\widehat{\mathbf{k}}\ni v \mapsto(v, v_2)\in W^{\perp}\subseteq \widehat{\mathbf{k}}\times\widehat{\mathbf{k}}\] with $v_2$ uniquely determined by the relation that $v(\frac{1}{\lambda}v_2)=1$. Then \[v_2(t)=\frac{1}{v(t\lambda^{-1})}=v(-t\lambda^{-1})=v(t(-\lambda)^{-1})=v(ts^{-1})=\phi(v)(t)\] for all $t\in \mathbf{k}$. Thus the group isomorphism above is precisely $(id,\phi)$ as given.

     Therefore, since any topological group isomorphism of compact groups preserves the normalized Haar measure, we deduce that $\nu=
    m_{W^{\perp}}=(id,\phi)_*m$. This finishes the proof of this Subcase 2.

    \textbf{Case 3.} $dim_{\mathbf{k}}W=2$.

    In this case, $W=V$. Then $W^{\perp}=\{(1,1)\}$, impossible due to the marginal condition.
\end{proof}

\begin{lemma}
The automorphism group of the algebraic action of $\mathbf{k}^*\curvearrowright\widehat{\mathbf{k}}$ is $\mathbf{k}^*$, i.e. we have $Aut_{\mathbf{k}^*}(\widehat{\mathbf{k}},m)\cong\mathbf{k}^*$ as topological groups.
\end{lemma}
\begin{proof}
    Consider the map $\mathbf{k}^*\ni s \overset{T}{\mapsto}\phi_s\in Aut_{\mathbf{k}^*}(\widehat{\mathbf{k}},m)$, where $\phi_s(v)(t)=v(ts^{-1})$ for all $t\in \mathbf{k}$ and $m$-a.e. $v\in\widehat{\mathbf{k}}$. Clearly, $T$ is an injective group homomorphism.
    
    Let $\phi\in\mathit{Aut_{\mathbf{k}^*}}(\widehat{\mathbf{k}},m)$. We now show that there exists $s\in \mathbf{k}^*$ such that  $\phi=\phi_s$. 
    
    Note that $\phi$ is a measurable conjugacy which commutes with the action $\mathbf{k}^*\curvearrowright\widehat{\mathbf{k}}$. The commuting condition tells us that $\phi(cv)=c\phi(v)$ for all $c\in \mathbf{k}^*$ and $m$-a.e. $v\in\widehat{\mathbf{k}}$. 

    Consider the map $\widehat{\mathbf{k}}\ni v \overset{(id,\phi)}{\mapsto}(v, \phi(v))\in\widehat{\mathbf{k}}\times\widehat{\mathbf{k}}$. This map is $\mathbf{k}^*$-equivariant since $\phi$ commutes with the action $\mathbf{k}^*\curvearrowright\widehat{\mathbf{k}}$. Thus $\nu:=(id,\phi)_*m$ is a $\mathbf{k}^*$-invariant ergodic probability measure on $\widehat{\mathbf{k}}\times\widehat{\mathbf{k}}$. Then we may apply the above mentioned Lemma 2.4 and Lemma 2.8 in \cite{BF} to deduce that $\nu=m_{W^{\perp}}$ for some linear subspace $W\subseteq\mathbf{k}\oplus \mathbf{k}$.

    Clearly, $W$ is dimension one and of the form $W=\{(c,c\lambda):c\in\mathbf{k}\}$ for some $\lambda\in \mathbf{k}^*$. Moreover, the proof of Subcase 2 in Lemma \ref{lem: 2-simple} shows that $\nu=(id,\psi)_*m$, where $\psi=\phi_s$ with $s=-\lambda$.

    Since $(id,\phi)_*m=\nu=(id,\psi)_*m$, $\phi$ and $\psi$ must be the same ($m$-a.e.). Indeed, for any measurable subset $E\subseteq \widehat{\mathbf{k}}$, we have that \[((id,\phi)_*m)(\phi^{-1}E\times E)=m((id,\phi)^{-1}(\phi^{-1}E\times E))=m(\phi^{-1}E\cap\phi^{-1}E)=m(\phi^{-1}E)\] and \[((id,\psi)_*m)(\phi^{-1}E\times E)=m((id,\psi)^{-1}(\phi^{-1}E\times E))=m(\phi^{-1}E\cap\psi^{-1}E).\] Thus $m(\phi^{-1}E\cap\psi^{-1}(E))=m(\phi^{-1}E)$. Similarly, we obtain that $m(\phi^{-1}E\cap\psi^{-1}E)=m(\psi^{-1}E)$. Note that $m(\phi^{-1}E)=m(E)=m(\psi^{-1}E)$ since $\phi$ and $\psi$ are measure-preserving, we deduce that $m(\phi^{-1}E\Delta\psi^{-1}E)=0$ for any measurable subset $E\subseteq \widehat{\mathbf{k}}$. Besides, it is easy to check that \[\{x\in \widehat{\mathbf{k}}: \phi(x)\neq \psi(x)\} \subseteq \bigcup_{n=1}^{\infty}(\phi^{-1}U_n\Delta\psi^{-1}U_n)\] for any countable basis $(U_n)_n$ of $\widehat{\mathbf{k}}$. Hence we have $\phi=\psi$ $m$-a.e.
    
    Thus $\phi$ corresponds to $s=-\lambda\in \mathbf{k}^*$. 
    
    Hence, $T$ is bijective and then $Aut_{\mathbf{k}^*}(\widehat{\mathbf{k}},m)$ is countable. 
    
    Since any countable Polish group must be discrete by Baire's category theorem, we deduce that  $Aut_{\mathbf{k}^*}(\widehat{\mathbf{k}},m)$ is discrete. Thus $T$ is in fact a homeomorphism.
\end{proof}

By Veech's theorem, i.e. \cite[Theorem 12.3]{Gla}, we deduce the key result which will be applied to prove the main theorem. 
\begin{corollary}\label{cor: classifying all factor maps}
    The only non-trivial factors of the algebraic action $\mathbf{k}^*\curvearrowright\widehat{\mathbf{k}}$ arise as quotients by non-trivial finite subgroups of $\mathbf{k}^*$. In other words, the only proper $\mathbf{k}^*$-invariant von Neumann subalgebras in $L(\mathbf{k})$ are of the form $L(\mathbf{k})^{F}$ for some non-trivial finite subgroup $F \subseteq\mathbf{k}^*$, where $L(\mathbf{k})^{F}=\{a\in L(\mathbf{k}):a=\sum_{q\in\mathbf{k}}c_qu_q, c_q=c_{qf}, \forall q\in \mathbf{k}, \forall f \in F\}$, the fixed point von Neumann subalgebras of $L(\mathbf{k})$ under the $F$-action.
\end{corollary}

\section{Invariant subalgebras inside $L(\mathbf{k}\rtimes \mathbf{k}^*)$}\label{section: invariant subalgebras}

Throughout this section, we fix  $\mathbf{k}$ to be a countable discrete field of characteristic zero. Set $E:=\{x\in\mathbf{k}: \text{$x$ is algebraic over $\mathbb{Q}$}\}$, let $\iota: E\hookrightarrow \mathbb{C}$ be an embedding of $E$ into $\mathbb{C}$. By replacing $E$ with $\iota(E)$, we may assume that $E\subseteq \mathbb{C}$. Note that for any $x\in \mathbf{k}$ with $x^m=1$ for some $m\geq 1$, $x$ is algebraic over $\mathbb{Q}$, thus $x\in \mathbb{C}$. Hence for any finite subgroup $F\subseteq\mathbf{k}^*$, we have that $F\subseteq \mathbb{C}$. Thus \cite[Lemma 6.1]{Mor} yields that $F$ is a cyclic subgroup generated by the primitive $m$-th root of unity $e^{2\pi i/m}$, where $m$ is the order of $F$.

For elements in $\mathbf{k}$, we usually  write them as $p,q$, while elements in $\mathbf{k}^*$ would be denoted by $r,s$.
We write $u_q$  (respectively $v_r$) for canonical unitaries in $L(\mathbf{k})$ (respectively, $L(\mathbf{k}^*)$), where $q\in\mathbf{k}$ and $r\in\mathbf{k}^*$,
so $v_ru_qv_r^*=u_{rq}$, $u_q^*=u_{-q}$ and $v_r^*=v_{1/r}$.

\begin{proposition}\label{prop: masas}
$L(\mathbf{k})$ and $L(\mathbf{k}^*)$ are maximal abelian von Neumann subalgebras (masas for short) in $L(\mathbf{k}\rtimes\mathbf{k}^*)$.  
\end{proposition}
\begin{proof}
This follows from \cite[Lemma 3.3.1]{masa_book}. Indeed, it is easy to check that for $H=\mathbf{k}$ or $\mathbf{k}^*$, we have that $\{hgh^{-1}: h\in H\}$ has infinite cardinality for each $g\in G\setminus H$, where $G=\mathbf{k}\rtimes \mathbf{k}^*$.  
\end{proof}

\begin{lemma}\label{lem: (L(k)^F)' cap L(G) = L(k) rtimes F}
    $[L(\mathbf{k})^F]'\cap L(\mathbf{k}\rtimes \mathbf{k}^*)=L(\mathbf{k})\rtimes F$, where $F$ is any finite subgroup of $\mathbf{k}^*$.
\end{lemma}
\begin{proof}
$\supseteq$: Note that $L(\mathbf{k})^{F}=\{{\sum_{f\in F}u_{qf}:q\in \mathbf{k}\}}''$ and $L(\mathbf{k})\rtimes F=L(\mathbf{k}\rtimes F)=\{u_p, v_r: p\in \mathbf{k}, r\in F\}''$. Thus we only need to show that $v_r(\sum_{f\in F}u_{qf})=(\sum_{f\in F}u_{qf})v_r$ for any $q\in \mathbf{k}$ and $r\in F$.

Since $r\in F$, a calculation shows that \[v_r(\sum_{f\in F}u_{qf})=(\sum_{f\in F}u_{qrf})v_r=(\sum_{f\in F}u_{qf})v_r.\]

$\subseteq$: We thank Prof. Hanfeng Li for pointing out the following concise proof. 

Take any $a\in [L(\mathbf{k})^F]'\cap L(\mathbf{k}\rtimes \mathbf{k}^*)=[L(\mathbf{k})^F]'\cap [L(\mathbf{k})\rtimes \mathbf{k}^*]$. Write $a=\sum_{s\in \mathbf{k}^*}f_sv_s$ for its Fourier expansion. where $f_s\in L(\mathbf{k})\subseteq \ell^2(\mathbf{k})$ and $v_s=\sigma_s\otimes \lambda_s\in B(\ell^2(\mathbf{k}
)\otimes \ell^2(\mathbf{k}^*))$. We need to prove that $f_s=0$ if $s\notin F$. 

For any $\xi \in L(\mathbf{k})^F$, we may compute the Fourier expansion of both sides of the identity $a\xi=\xi a$: \[a\xi =(\sum_{s\in \mathbf{k}^*}f_sv_s)\xi=\sum_{s\in \mathbf{k}^*}(f_s\sigma_s(\xi))v_s,\] and \[\xi a=\sum_{s\in \mathbf{k}^*}(\xi f_s)v_s.\]
Thus, comparing the coefficients of $v_s$, we obtain that $f_s(\sigma_s(\xi)-\xi)=0$.

Note that $\mathbf{k}$ is certainly a torsion-free abelian and hence elementary amenable group. Moreover, for $\xi:=\sum_{f\in F}u_f\in \mathbb{C}[\mathbf{k}]\cap L(\mathbf{k})^F$, we have $\sigma_s(\xi)-\xi\neq 0$ for all $s\not\in F$. Thus, by \cite[Theorem 2]{Lin}, we obtain the desired result.
\end{proof}

\begin{remark} Below is an alternative dynamical proof of ``$\subseteq$'' part of the lemma above, which was placed in the main text in a previous version.
\begin{proof}
    Let $\pi:X\rightarrow (Y,\nu)$ be the factor map corresponding to the $\mathbf{k}^*$-invariant von Neumann subalgebra $L(\mathbf{k})^F$, where $X=\widehat{\mathbf{k}}$. Then the claim is equivalent to $L^\infty(Y)'\cap(L^\infty(X)\rtimes \mathbf{k}^*)\subseteq L^\infty(X)\rtimes F$. 

Let $E: L^\infty(X,m)\rightarrow L^\infty(Y,\nu)$ be the conditional expectation. Then $E(f)(y)=\int_Xf(x)\,d\mu_y(x)$, where $f\in L^\infty(X,m)$ and $m=\int_Y\mu_y\,d\nu(y)$ is the measure decomposition with respect to $\pi$. Note that $E$ satisfies $E|_{L^\infty(Y,\nu)}=id$ and $E(f\xi)=E(f)\xi$ for all $f\in L^\infty(X,m)$, $\xi \in L^\infty(Y,\nu)$. Note that $E$ is also faithful, i.e. if $f\geq0$ in $L^\infty(X,m)$ and $E(f)=0$, then $f=0$. Indeed, $E(f)=0$ implies that $\int_Xf(x)\,d\mu_y(x)=0$ for $\nu$-a.e. $y\in Y$. Since $f(x)\geq0$ for $m$-a.e. $x\in X$. We deduce that $f(x)=0$ for $\mu_y$-a.e. $x\in X$ and $\nu$-a.e. $y\in Y$. Then $m(\{x:f(x)\neq0\})=\int_Y \mu_y(\{x:f(x)\neq0\})\,d\nu(y)=0$, which implies $f=0$. 

Take any $a\in L^\infty(Y)'\cap(L^\infty(X)\rtimes \mathbf{k}^*)$. Write $a=\sum_{s\in \mathbf{k}^*}f_sv_s$ for its Fourier expansion, where $f_s\in L^\infty(X,m)$ and $v_s=\sigma_s\otimes \lambda_s\in B(L^2(X,m)\otimes \ell^2(\mathbf{k}^*))$. We need to prove that $f_s=0$ if $s\notin F$. 

For any $\xi \in L^\infty(Y,\nu)$, the same argument as in the main text shows that $f_s(\sigma_s(\xi)-\xi)=0$, hence $f_s^*f_s(\sigma_s(\xi)-\xi)=0$. Taking $E$ on both sides, we deduce that $0=E(f_s^*f_s(\sigma_s(\xi)-\xi))=E(f_s^*f_s)(\sigma_s(\xi)-\xi)$ for all $s\in \mathbf{k}^*$ and $\xi \in L^\infty(Y,\nu)$. 

Since $s\notin F$, we may check that $s$ acts essentially freely on $(Y,\nu)$, i.e. the set $\{y\in Y: sy=y\}$ of points fixed by $s$ is of measure zero. Indeed, for any $x\in X$, let $[x]$ denote the equivalence class of $x$ in $Y$. If $s[x]=[x]$, then $sx=rx$ for some $r\in F$, which implies $x(s^{-1}q)=x(r^{-1}q)$ for all $q\in \mathbf{k}$. Hence $x((s^{-1}-r^{-1})q)=1$ for all $q\in \mathbf{k}$. Since $s\notin F$, we deduce that $\{(s^{-1}-r^{-1})q:q\in \mathbf{k}\}=\mathbf{k}$, thus $x\equiv1$. Therefore, there is only one point $[1]\in Y$ which is fixed by $s$. Note that $\pi^{-1}\{[1]\}=\{1\}\subseteq\widehat{\mathbf{k}}$,  thus $\nu\{[1]\}=\pi_*m\{[1]\}=m\{1\}$. Since $\mathbf{k}$ is an infinite discrete field, the Haar measure $m$ on $\widehat{\mathbf{k}}$ is non-atomic (otherwise $\widehat{\mathbf{k}}$ is discrete and then $\mathbf{k}$ is compact, contradicting to $\mathbf{k}$ is infinite discrete), which implies $m\{1\}=0$. Thus the set $\{[1]\}$ of points fixed by $s$ is of measure zero. 

Hence, we get that $E(f_s^*f_s)=0$. Indeed, otherwise, we have $\nu(U)>0$ for $U=\{y\in Y: E(f_s^*f_s)(y)\neq0\}$. By \cite[Proposition A.22]{KL_book}, there exists a measurable subset $B\subseteq U$ such that $\nu(B)>0$ and $sB\cap B=\varnothing$. Let $\xi=\chi_B$. Then for any $y\in B$, $E(f_s^*f_s)(y)(\sigma_s(\xi)-\xi)(y)=E(f_s^*f_s)(y)(\chi_B(s^{-1}y)-\chi_B(y))=-E(f_s^*f_s)(y)\neq0$, contradicting to $E(f_s^*f_s(\sigma_s(\xi)-\xi))=0$.

Since $E$ is faithful, we deduce that $f_s^*f_s=0$, thus $f_s=0$. This finishes the proof.
\end{proof}
\end{remark}

\begin{lemma}\label{lem: invariant A contains L(k)}
    Assume that $A$ is a $\mathbf{k}\rtimes \mathbf{k}^*$-invariant von Neumann subalgebra such that $L(\mathbf{k})\subseteq A\subseteq L(\mathbf{k}\rtimes \mathbf{k}^*)$ and the center $Z(A)=L(\mathbf{k})^F$ for a non-trivial finite (cyclic) subgroup $F\subseteq\mathbf{k}^*$. Then $A=L(\mathbf{k}\rtimes F)$.
\end{lemma}
\begin{proof}
First, by Lemma \ref{lem: (L(k)^F)' cap L(G) = L(k) rtimes F}, we deduce that $A\subseteq Z(A)'\cap L(\mathbf{k}\rtimes \mathbf{k}^*)= L(\mathbf{k}\rtimes F)$.

Denote by $E: L(\mathbf{k}\rtimes \mathbf{k}^*)\rightarrow A$ the trace preserving conditional expectation.

Using the $\mathbf{k}\rtimes \mathbf{k}^*$-invariance of $A$, we deduce that $E(v_s)\in L(\mathbf{k}^*)'\cap L(\mathbf{k}\rtimes \mathbf{k}^*)=L(\mathbf{k}^*)$ for all $s\in \mathbf{k}^*$ by Proposition \ref{prop: masas}.
Hence $E(v_s)\in L(\mathbf{k}^*)\cap A\subseteq L(\mathbf{k}^*)\cap L(\mathbf{k}\rtimes F)\subseteq L(F)$. Write $F=\langle z\rangle$ with $|z|=1$ and $z$ is a primitive $m$-th root of $1$, where $m$ is the order of $F$.

Since $L(\mathbf{k})\subseteq A \subseteq L(\mathbf{k}\rtimes F)$, we have that $A=\{E(b): b\in L(\mathbf{k}\rtimes F)\}''=\{L(\mathbf{k}), E(v_{z^i}): 0\leq i\leq m-1\}''$.

For any $0\leq k\leq m-1$, since $E(v_{z^k})\in L(F)$, we may write $E(v_{z^k})=\sum_{\ell=0}^{m-1}c_{\ell}^{(k)}v_{z^{\ell}}$, where $c_{\ell}^{(k)}\in\mathbb{C}$.

For any $0\neq h\in \mathbf{k}$, we have 
\begin{align*}
    A\ni u_hE(v_{z^k})u_h^*
    =E(u_hv_{z^k}u_h^*)=E(u_{h-z^kh}v_{z^k}).
\end{align*}
On the one hand, we have 
\begin{align}\label{eq-contain-L(k)-1}
u_hE(v_{z^k})u_h^*=\sum_{\ell=0}^{m-1}c_{\ell}^{(k)}u_hv_{z^{\ell}}u_h^*=\sum_{\ell=0}^{m-1}c_{\ell}^{(k)}u_{h-z^{\ell}h}v_{z^{\ell}}.
\end{align}
On the other hand,
\begin{align}\label{eq-contain-L(k)-2}
    u_hE(v_{z^k})u_h^*
    &=E(u_{h-z^kh}v_{z^k}) \notag\\
    &=E(E(u_{h-z^kh}v_{z^k})) \notag\\
    &\overset{\eqref{eq-contain-L(k)-1}}{=}E(\sum_{\ell=0}^{m-1}c_{\ell}^{(k)}u_{h-z^{\ell}h}v_{z^{\ell}}) \notag\\
    &=\sum_{\ell=0}^{m-1}c_{\ell}^{(k)}u_{h-z^{\ell}h}E(v_{z^{\ell}}) ~~(\text{since $L(\mathbf{k})\subseteq A$}) \notag\\
    &=\sum_{\ell=0}^{m-1}c_{\ell}^{(k)}u_{h-z^{\ell}h}\sum_{\ell'=0}^{m-1}c_{\ell'}^{(\ell)}v_{z^{\ell'}} \notag\\ 
    &=\sum_{\ell,\ell'=0}^{m-1}c_{\ell}^{(k)}c_{\ell'}^{(\ell)}u_{h-z^{\ell}h}v_{z^{\ell'}}.
\end{align}
Then we compare the coefficients using \eqref{eq-contain-L(k)-1} and \eqref{eq-contain-L(k)-2}.

First, look at the coefficient of $u_{h-z^ih}v_{z^i}$ for any fixed $0\leq i\leq m-1$, we deduce that
$c_i^{(k)}=c_i^{(k)}c_i^{(i)}$ for any $0<i,k\leq m-1$.

\textbf{Case 1.} If for all $0<i\leq m-1$ we have $c_i^{(i)}\neq 1$. Then $c_i^{(k)}=0$ for all $0<i, k\leq m-1$. Hence $E(v_{z^k})=c_0^{(k)}v_{z^0}=c_0^{(k)}$. By computing the trace, we deduce that $E(v_{z^k})=0$ for all $0<k\leq m-1$. Thus $A=L(\mathbf{k})$, contradicting to $Z(A)=L(\mathbf{k})^F$.

\textbf{Case 2.} There exists some $0<i\leq m-1$ such that $c_i^{(i)}=1$. Then $\tau(v_{z^i}^*E(v_{z^i}))=c_i^{(i)}=1$, i.e. $\|E(v_{z^i})-v_{z^i}\|_2^2=0$. Hence $v_{z^i}=E(v_{z^i})\in A$. Pick the minimal $d$ such that $0<d\leq m-1$ 
 and $v_{z^d}\in A$. 

 We claim that $A=L(\mathbf{k})\rtimes \langle z^d\rangle$.

To show this, we still compare the coefficients of both sides of \eqref{eq-contain-L(k)-1}
 and \eqref{eq-contain-L(k)-2}.

Given any $\ell\neq \ell'$ and look at the coefficient of $u_{h-z^{\ell}h}v_{z^{\ell'}}$, we deduce that 
\begin{align}\label{eq-contain-L(k)-3}
0=c_{\ell}^{(k)}c_{\ell'}^{(\ell)}.
\end{align}

Now assume that $v_{z^{\ell}}\not\in A$ for some $\ell$, i.e. $E(v_{z^{\ell}})\neq v_{z^{\ell}}$. Then by looking at the Fourier expansion of $E(v_{z^{\ell}})$, we have  either   $\exists \ell'\neq \ell$ such that $c_{\ell'}^{(\ell)}\neq 0$ or $E(v_{z^{\ell}})=0$.
In the first case, we deduce that $c_{\ell}^{(k)}=0$ for all $k$ by \eqref{eq-contain-L(k)-3}. In particular, $0=c_{\ell}^{(\ell)}=\tau(v_{z^{\ell}}^*E(v_{z^{\ell}}))=\|E(v_{z^{\ell}})\|_2^2$. Hence $E(v_{z^{\ell}})=0$. So in either case, we have $E(v_{z^{\ell}})=0$.

This shows that for any $0\leq \ell\leq m-1$, we have either $E(v_{z^{\ell}})=0$ or $v_{z^{\ell}}$.
Clearly, this shows that $A=L(\mathbf{k})\rtimes \langle z^d\rangle$ since $A=\{L(\mathbf{k}), E(v_{z^i}): 0\leq i\leq m-1\}''$. 

Finally, note that a calculation shows that $Z(A)=L(\mathbf{k})^{\langle z^d\rangle}$. Thus the assumption on $Z(A)$ implies that in fact we have that $\langle z^d\rangle=F$. Hence $A=L(\mathbf{k}\rtimes F)$.
\end{proof}

\begin{lemma}\label{lem: difference between two complex numbers with modulus one}
    Let $x,y,z$ be complex numbers with modulus one. Assume that $1-x=y-z$. Then either $x=1$ or $y=1$ or $y=-x$.
\end{lemma}
\begin{proof}
Write $x=e^{i\alpha}$ and $y=e^{i\beta}$.
Since $z=x+y-1$ has modulus one, we have that $1=|e^{i\alpha}+e^{i\beta}-1|$. A calculation shows that this boils down to 
\begin{align*}
    (cos\alpha-1)(cos\beta-1)+sin\alpha sin\beta=0
.
\end{align*}
Using $cos\alpha-1=-2sin^2(\frac{\alpha}{2})$ , $cos\beta-1=-2sin^2(\frac{\beta}{2})$ and $sin\alpha=2sin(\frac{\alpha}{2})cos(\frac{\alpha}{2})$ and $sin\beta=2sin(\frac{\beta}{2})cos(\frac{\beta}{2})$, we deduce that
\begin{align*}
    sin(\frac{\alpha}{2})sin(\frac{\beta}{2})(sin(\frac{\alpha}{2})sin(\frac{\beta}{2})+cos(\frac{\alpha}{2})cos(\frac{\beta}{2}))=0,
\end{align*}
this shows that
\begin{align*}
    sin(\frac{\alpha}{2})sin(\frac{\beta}{2})cos(\frac{\alpha-\beta}{2})=0
.
\end{align*}
Hence either $sin(\frac{\alpha}{2})=0$ or $sin(\frac{\beta}{2})=0$ or $cos(\frac{\alpha-\beta}{2})=0$. This is equivalent to either $x=1$ or $y=1$ or $y=-x$.
\end{proof}

\begin{lemma} \label{lem: lemma about theta, k', delta}
    Let $m\geq 1$. Let $z=e^{2\pi i/m}$ be a primitive $m$-th root of 1. For given positive integers $0<\theta, k', \delta <m$, if there exists an integer $k$ such that 
\begin{align*}
    {z^k(1-z^{k'})}(1-z^{\theta})=(1-z^{k+k'})z^{\theta-k'}(1-z^{\delta}).
\end{align*}
Then either $\delta=k'$ and $k\equiv \theta-k'(\text{mod}~ m)$ or $\delta+\theta=m$ and $k\equiv -2k'(\text{mod}~m)$. 

\end{lemma}
\begin{proof}
    For integers $0<\theta,k',\delta<m$, we consider when there exists an integer $k$ such that
\[
z^k(1-z^{k'})(1-z^{\theta})=(1-z^{k+k'})z^{\theta-k'}(1-z^{\delta}). \tag{1}
\]
in several steps.

\textbf{Step 1.} Divide (1) by $z^{\theta-k'}\neq0$ and set $A=k+k'$, then we have
\[
z^{A-\theta}(1-z^{k'})(1-z^{\theta})=(1-z^{A})(1-z^{\delta}). \tag{2}
\]

\textbf{Step 2.} Observe that for a positive integer $a$, we have $1-z^a=-2ie^{\pi i a/m}\sin(\pi a/m)$. Substitute into (2) gives
\[
z^{A-\theta}(-2i)^2e^{\pi i(k'+\theta)/m}\sin\frac{\pi k'}{m}\sin\frac{\pi\theta}{m}
=(-2i)^2e^{\pi i(A+\delta)/m}\sin\frac{\pi A}{m}\sin\frac{\pi\delta}{m}.
\]
Since $z^{A-\theta}=e^{2\pi i(A-\theta)/m}$, the left-hand side becomes
\[
-4e^{2\pi iA/m}e^{\pi i(k'-\theta)/m}\sin\frac{\pi k'}{m}\sin\frac{\pi\theta}{m}.
\]
The right-hand side is
\[
-4e^{\pi i(A+\delta)/m}\sin\frac{\pi A}{m}\sin\frac{\pi\delta}{m}.
\]
Canceling $-4$ and dividing by $e^{\pi i(A+\delta)/m}$, we deduce that
\[
e^{\pi i(A+k'-\theta-\delta)/m}
=\frac{\sin(\pi A/m)\sin(\pi\delta/m)}{\sin(\pi k'/m)\sin(\pi\theta/m)}. \tag{3}
\]

\textbf{Step 3.} The left side of (3) has modulus 1 while the right side is real. Hence
\[
e^{\pi i(A+k'-\theta-\delta)/m}=\pm1,
\]
so
\[
\frac{A+k'-\theta-\delta}{m}\in\mathbb{Z},\qquad\text{i.e.}\qquad
A\equiv\theta+\delta-k'\pmod{m}.
\]
Write $A=\theta+\delta-k'+mn$ with $n\in\mathbb{Z}$. Then the left side of (3) equals $(-1)^n$. Also
\[
\sin\frac{\pi A}{m}=\sin\Bigl(\pi\frac{\theta+\delta-k'}{m}+\pi n\Bigr)=(-1)^n\sin\Bigl(\pi\frac{\theta+\delta-k'}{m}\Bigr).
\]
Substitute into (3) and cancel $(-1)^n$, we deduce that
\[
\sin\Bigl(\pi\frac{\theta+\delta-k'}{m}\Bigr)\sin\frac{\pi\delta}{m}
=\sin\frac{\pi k'}{m}\sin\frac{\pi\theta}{m}. \tag{4}
\]

\textbf{Step 4.} Set $\alpha=\pi\theta/m$, $\beta=\pi k'/m$, $\gamma=\pi\delta/m$ (all in $(0,\pi)$). Then (4) becomes
\[
\sin(\alpha+\gamma-\beta)\sin\gamma=\sin\beta\sin\alpha.
\]
Using $\sin X\sin Y=\frac12[\cos(X-Y)-\cos(X+Y)]$, we deduce that
\[
\cos(\alpha-\beta)-\cos(\alpha+2\gamma-\beta)=\cos(\alpha-\beta)-\cos(\alpha+\beta).
\]
Cancel $\cos(\alpha-\beta)$ and multiply by $-1$, then
\[
\cos(\alpha+2\gamma-\beta)=\cos(\alpha+\beta).
\]
Hence
\[
\alpha+2\gamma-\beta=\pm(\alpha+\beta)+2\pi t,\quad t\in\mathbb{Z}.
\]

\textbf{Case +:} $\alpha+2\gamma-\beta=\alpha+\beta+2\pi t\Rightarrow\gamma=\beta+\pi t$. Since $\gamma,\beta\in(0,\pi)$, $t=0$ gives $\gamma=\beta$, i.e. $\delta=k'$. It is not hard to check that we have $k\equiv \theta-k'(\text{mod}~ m)$.

\textbf{Case --:} $\alpha+2\gamma-\beta=-\alpha-\beta+2\pi t\Rightarrow2\alpha+2\gamma=2\pi t\Rightarrow\alpha+\gamma=\pi t$. Because $\alpha,\gamma\in(0,\pi)$, we have $\alpha+\gamma\in(0,2\pi)$, so $t=1$ and $\alpha+\gamma=\pi$, i.e. $\theta+\delta=m$.
We can check that $k\equiv -2k'(\text{mod}~m)$.

Hence we finish the proof.
\end{proof}

The following might be the most difficult step for the whole proof, which extends Lemma \ref{lem: invariant A contains L(k)} by dropping the condition that $L(\mathbf{k})\subseteq A$. It shows that the invariance condition drastically restricts the structure of intermediate subalgebras. We do not know whether there is a more conceptual proof or not, see Remark \ref{remark: direct integral gives more intermediate ones} for more comments on this.

\begin{lemma} \label{lem: invariant A with Z(A)=L(k)^{F}}
    Let $A\subseteq L(\mathbf{k}\rtimes \mathbf{k}^*)$ be a $\mathbf{k}\rtimes \mathbf{k}^*$-invariant von Neumann subalgebra such that $Z(A)=L(\mathbf{k})^F$ for some non-trivial finite (cyclic) subgroup $F\subseteq \mathbf{k}^*$. Then either $A=L(\mathbf{k})^{F}$ or $A=L(\mathbf{k}\rtimes F)$.
\end{lemma}
\begin{proof}
Write $F=\langle z\rangle$ with $|z|=1$ and $z=e^{2\pi i/m}$ is a primitive $m$-th root of $1$ for some $m\geq 2$.

Note that $A\subseteq Z(A)'\cap L(\mathbf{k}\rtimes \mathbf{k}^*)=[L(\mathbf{k})^F]'\cap L(\mathbf{k}\rtimes \mathbf{k}^*)=L(\mathbf{k}\rtimes F)$ by Lemma \ref{lem: (L(k)^F)' cap L(G) = L(k) rtimes F}.

Denote by $E: L(\mathbf{k}\rtimes \mathbf{k}^*)\rightarrow A$ the trace preserving conditional expectation.

Using the $\mathbf{k}\rtimes \mathbf{k}^*$-invariance of $A$, we deduce that $E(u_q)\in L(\mathbf{k})'\cap L(\mathbf{k}\rtimes \mathbf{k}^*)=L(\mathbf{k})$ for all $q\in\mathbf{k}$ and $E(v_s)\in L(\mathbf{k}^*)'\cap L(\mathbf{k}\rtimes \mathbf{k}^*)=L(\mathbf{k}^*)$ for all $s\in \mathbf{k}^*$  by Proposition \ref{prop: masas}. Thus, similar to the argument in Lemma \ref{lem: invariant A contains L(k)}, we deduce that $E(v_s)\in L(F)$ for all $s\in \mathbf{k}^*$.

For any $0\leq k\leq m-1$, write $E(v_{z^k})=\sum_{\ell=0}^{m-1}c_{\ell}^{(k)}v_{z^{\ell}}$, where $c_{\ell}^{(k)}\in\mathbb{C}$. Clearly, by taking trace on both sides, we have that $c_0^{(k)}=0$ for all $0<k\leq m-1$.

Note that for any $q\in \mathbf{k}$, we have 
\begin{align*}
E(u_qv_{z^k}u_q^*)=u_qE(v_{z^k})u_q^*=\sum_{j=0}^{m-1}c_j^{(k)}u_{q-qz^j}v_{z^j}.
\end{align*}

Fix any $0<k,k'\leq m-1$ and any $0\neq q\in \mathbf{k}$, let us compute $E(E(v_{z^k})u_qv_{z^{k'}}u_q^*)$ in two ways.

On the one hand, we have
\begin{align}\label{eq: k-invariance-eq-1}
E(E(v_{z^k})u_qv_{z^{k'}}u_q^*)
    &=E(v_{z^k})E(u_qv_{z^{k'}}u_q^*)  \notag\\ 
    &=(\sum_{i=0}^{m-1}c_{i}^{(k)}v_{z^i})(\sum_{j=0}^{m-1}c_j^{(k')}u_{q-qz^j}v_{z^{j}}) \notag\\ 
    &=\sum_{i,j=0}^{m-1}c_i^{(k)}c_j^{(k')}u_{z^iq(1-z^j)}v_{z^{i+j}}.
\end{align}
On the other hand, note that for $i\neq m-k'$, the term $1-z^{i+k'}\neq0$, allowing us to safely define $\tilde{q_i}=\frac{z^iq(1-z^{k'})}{1-z^{i+k'}}$ below and split the sum into two parts. Thus we have
\begin{align}\label{eq: k-invariance-eq-2}
&E(E(v_{z^k})u_qv_{z^{k'}}u_q^*)  \notag \\ 
   &=E(\sum_{i=0}^{m-1}c_i^{(k)}v_{z^i}u_{q(1-z^{k'})}v_{z^{k'}}) \notag\\ 
   &=\sum_{i=0}^{m-1}c_i^{(k)}E(u_{z^iq(1-z^{k'})}v_{z^{i+k'}}) \notag\\ 
   &=\sum_{i=0,i\neq m-k'}^{m-1}c_i^{(k)}E(u_{z^iq(1-z^{k'})}v_{z^{i+k'}})+c_{m-k'}^{(k)}E(u_{z^{m-k'}q-q}) \notag\\ 
   &=\sum_{i=0,i\neq m-k'}^{m-1}c_i^{(k)}u_{\tilde{q_i}}E(v_{z^{i+k'}})u_{\tilde{q_i}}^*+c_{m-k'}^{(k)}E(u_{z^{m-k'}q-q}) \notag\\ 
   &=\sum_{i=0,i\neq m-k'}^{m-1}c_i^{(k)}(\sum_{j=0}^{m-1}c_j^{(i+k')}u_{\tilde{q_i}}v_{z^j}u_{\tilde{q_i}}^*)+c_{m-k'}^{(k)}E(u_{z^{m-k'}q-q}) \notag\\ 
   &=\sum_{i=0,i\neq m-k'}^{m-1}\sum_{j=0}^{m-1}c_i^{(k)}c_j^{(i+k')}u_{\tilde{q_i}(1-z^j)}v_{z^j}+c_{m-k'}^{(k)}E(u_{z^{m-k'}q-q}).
\end{align}
Next, we compare the coefficients of $u_{z^{\theta-k'}(q-z^{\delta}q)}v_{z^{\theta}}$ for any $0<\theta,k',\delta<m$.

First, we compute the corresponding coefficient from \eqref{eq: k-invariance-eq-1}. Set
$z^iq(1-z^j)=z^{\theta-k'}q(1-z^{\delta})$ and $z^{i+j}=z^{\theta}$. We deduce that
$1-z^{i-\theta}=z^{\delta-k'}-z^{-k'}$. Applying Lemma \ref{lem: difference between two complex numbers with modulus one}, we deduce that there are three possibilities (note that most of the following equalities e.g. $i=\theta-k'$ should be understood as congruences $\text{mod}~m$):
\begin{itemize}
    \item[(a)] $i-\theta=0$, $\delta-k'=-k'$. Hence $i=\theta$, $j=0$. 
    \item[(b)] $\delta=k'$, $i-\theta=-k'$. Hence $i=\theta-k'$, $j=k'$.
    \item[(c)] $z^{\delta-k'}=-z^{i-\theta}$, $1=-z^{-k'}$, i.e. $z^{k'}=-1$ and $\delta=i-\theta$. Hence $i=\delta+\theta$, $j=m-\delta$.
\end{itemize}

Recall that $c^{(k')}_0=0$ for all $0<k'\leq m-1$.
Therefore, the coefficient we get from \eqref{eq: k-invariance-eq-1} is equal to the following expression (the case that $\delta=k'$ and $z^{k'}=-1$ is not considered here, since it will not be required later. Moreover, the subscripts e.g. $\delta+\theta$ in $c^{(k)}_{\delta+\theta}$ should be understood as the corresponding number between $0$ and $m$ after mod $m$):  
\begin{align}
\begin{cases}\label{eq: coefficent on LHS of the key lemma}
 c_{\theta-k'}^{(k)}c_{k'}^{(k')},&\text{if}~\delta=k',z^{k'}\neq -1\\ 
c_{\delta+\theta}^{(k)}c_{m-\delta}^{(k')},&\text{if}~\delta\neq k',z^{k'}=-1\\ 
0,&\text{if}~\delta\neq k',z^{k'}\neq -1.
    \end{cases}
\end{align}

To compute the corresponding coefficient from $\eqref{eq: k-invariance-eq-2}$, we first note that $E(u_{z^{m-k'}q-q})\in L(\mathbf{k})$ and hence has no contribution as $\theta\neq 0,m$, therefore, the corresponding coefficient equals $c_i^{(k)}c_j^{(i+k')}$, where $j=\theta$ and 
\begin{align*}
    \tilde{q_i}(1-z^j)=\tilde{q_i}(1-z^{\theta})=z^{\theta-k'}(q-z^{\delta}q).
\end{align*}
Plugging the expression for $\tilde{q_i}$ into the above identity and simplifying it, we  deduce that
\begin{align*}
    \frac{z^i(1-z^{k'})}{1-z^{i+k'}}(1-z^{\theta})=z^{\theta-k'}(1-z^{\delta}).
\end{align*}

Then  Lemma \ref{lem: lemma about theta, k', delta} yields either $\delta=k'$ and $i\equiv \theta-k'(\text{mod}~ m)$ or $\delta+\theta=m$ and $i\equiv -2k'(\text{mod}~m)$. Also note that for $i\equiv \theta-k'(\text{mod}~ m)$ or $i\equiv -2k'(\text{mod}~m)$, it automatically satisfies that $i\neq m-k'$ once we take $0<\theta,k'<m$.
Therefore, the coefficient from $\eqref{eq: k-invariance-eq-2}$ equals the following expression (once again, the subscripts e.g. $-2k'$ in $c_{-2k'}^{(k)}$ should be understood as the corresponding number between $0$ and $m$ after mod $m$):
\begin{align}\label{eq: coefficent on RHS of the key lemma}
    \begin{cases}
        c_{\theta-k'}^{(k)}c_{\theta}^{(\theta)},&\text{if}~\delta=k'\\ 
        c_{-2k'}^{(k)}c_{\theta}^{(m-k')},&\text{if}~\delta=m-\theta, \\ 
        0,&\text{if}~\delta\neq k',m-\theta.
    \end{cases}
\end{align}
Note that if $\delta=k'=m-\theta$, then the above first two cases coincide.

Let us continue the whole proof, and we will compare \eqref{eq: coefficent on LHS of the key lemma} with \eqref{eq: coefficent on RHS of the key lemma} several times for suitable chosen $0<k<m$  and $0<\theta,\delta,k'<m$ with $k'\leq \theta$ and $k'\leq \frac{m}{2}$. 

Assume that there exists some $1\leq k\leq m-1$ such that $E(v_{z^k})=v_{z^k}$.
    Then $v_{z^k}\in A$ and hence for all $q\in\mathbf{k}$, we have $u_{q-qz^k}v_{z^k}=u_qv_{z^k}u_q^*\in A$ as $A$ is $\mathbf{k}$-invariant, thus $u_{q-qz^k}\in A$ for all $q\in \mathbf{k}$. Since $\{q-qz^k: q\in\mathbf{k}\}=\mathbf{k}$, we deduce that $L(\mathbf{k})\subseteq A$. Thus we can apply Lemma \ref{lem: invariant A contains L(k)} to deduce that $A=L(\mathbf{k})\rtimes F$.

So to finish the proof, we may assume that $E(v_{z^k})\neq v_{z^k}$ for all $1\leq k\leq m-1$ from now on.

\textbf{Claim.} $E(v_{z^k})=0$ for all $1\leq k\leq m-1$.

\textbf{Proof of the claim.} 
Note that we may assume $m\geq 3$.
Indeed, if $m=2$, then $E(v_z)=c^{(1)}_0+c_1^{(1)}v_z=c_1^{(1)}v_z$. By taking trace on both sides, we deduce that $c^{(1)}_1v_z=E(v_z)=E(E(v_z))=c^{(1)}_1E(v_z)=[c^{(1)}_1]^2v_z$. Thus $c_1^{(1)}=[c_1^{(1)}]^2$, i.e. $c_1^{(1)}=0$ or $1$. Since we assume $E(v_z)\neq v_z$, we deduce that $c_1^{(1)}=0$, i.e. $E(v_z)=0$.

From now on, we assume that $m\geq 3$, and we split the proof by considering the parity of $m$.

\textbf{Case 1.} $m$ is even.

Note that in this case $z^{k'}=-1$ iff $k'=\frac{m}{2}$.

Now assume the claim fails, so there exists some  $0<k'<m$ such that $E(v_{z^{k'}})\neq 0$.
Since $v_{z^{k'}}=v_{z^{m-k'}}^*$ and $E(v_{z^{m-k'}})=E(v_{z^{k'}}^*)=E(v_{z^{k'}})^*$, we may assume that $0<k'\leq \frac{m}{2}$.

\textbf{Subcase 1}. $0<k'<\frac{m}{2}$.

\textbf{Subsubcase 1.} For all $\theta\neq m-k'$, $c_{\theta}^{(m-k')}=0$.

Then $E(v_{z^{m-k'}})=c^{(m-k')}_{m-k'}v_{z^{m-k'}}$. By taking $E$ on both sides, it is not hard to deduce that either $E(v_{z^{m-k'}})=0$ or $E(v_{z^{m-k'}})=v_{z^{m-k'}}$. Note that $E(v_{z^{m-k'}})=E(v_{z^{k'}}^*)=E(v_{z^{k'}})^*$. This  also means that $E(v_{z^{k'}})=0$ or $E(v_{z^{k'}})=v_{z^{k'}}$.

However, recall that we have $E(v_{z^k})\neq v_{z^k}$ for all $0<k<m$ and $E(v_{z^{k'}})\neq 0$, the above subsubcase 1 does not hold for this $k'$. Therefore, we are in the following Subsubcase 2.

\textbf{Subsubcase 2.} There exists some $\theta\neq m-k'$ such that $c_{\theta}^{(m-k')}\neq 0$. 

Set $\delta=m-\theta$. Note that $0<\delta<m$ and $\delta\neq k'$. Since $k'\neq \frac{m}{2}$, by comparing \eqref{eq: coefficent on LHS of the key lemma} with \eqref{eq: coefficent on RHS of the key lemma}, we deduce that
$0=c_{m-2k'}^{(k)}c_{\theta}^{(m-k')}$. Then we deduce that $c_{m-2k'}^{(k)}=0$ for all $0<k<m$. We may set $k=m-2k'$ to deduce that $c^{(m-2k')}_{m-2k'}=0$. Note that $\|E({v_{z^{m-2k'}}})\|_2^2=\tau(E(v_{z^{m-2k'}}^*)E(v_{z^{m-2k'}}))=\tau(v_{z^{m-2k'}}^*E(v_{z^{m-2k'}}))=c^{(m-2k')}_{m-2k'}$. Hence $E(v_{z^{m-2k'}})=0$. Thus $E(v_{z^{2k'}})=E(v_{z^{m-2k'}}^*)=0$.

So based on the above argument, we deduce that $E(v_{z^{2k'}})=0$.

Next, we will draw a contradiction.

Take $\delta=k'$ and note that we assumed that $k'\neq \frac{m}{2}$, by comparing \eqref{eq: coefficent on LHS of the key lemma} with \eqref{eq: coefficent on RHS of the key lemma}, we deduce that $c_{\theta-k'}^{(k)}c_{k'}^{(k')}=c_{\theta-k'}^{(k)}c_{\theta}^{(\theta)}$, i.e. $c_{\theta-k'}^{(k)}[c_{k'}^{(k')}-c_{\theta}^{(\theta)}]=0$. Set $0<\theta:=2k'<m$, then 
$0=c_{k'}^{(k)}[c_{k'}^{(k')}-c_{2k'}^{(2k')}]$.
From $E(v_{z^{2k'}})=0$ and $E(v_{z^{k'}})\neq 0$, we deduce that $c_{2k'}^{(2k')}=0$ while $c_{k'}^{(k')}=\|E({v_{z^{k'}}})\|_2^2\neq 0$. Hence  $0=c_{k'}^{(k)}$. Set $k=k'$, we deduce that $0=c_{k'}^{(k')}=\|E(v_{z^{k'}})\|_2^2$. i.e. $E(v_{z^{k'
}})=0$, a contradiction.

\textbf{Subcase 2.} $k'=\frac{m}{2}$.

Take any $0<\delta\neq \frac{m}{2}<m$ (such $\delta$ exists since we have assumed that $m\geq 3$), then by comparing \eqref{eq: coefficent on LHS of the key lemma} with \eqref{eq: coefficent on RHS of the key lemma}, we deduce that
\begin{align*}
        c_{\theta+\delta}^{(k)}c_{m-\delta}^{(\frac{m}{2})}=    \begin{cases}
c_{0}^{(k)}c_{\theta}^{(\frac{m}{2})}=0,~&\text{if}~\delta=m-\theta\\ 
0,~&\text{if}~\delta\neq m-\theta
    \end{cases}
    =0.
\end{align*}
Set $k=\frac{m}{2}$. Then $c^{(\frac{m}{2})}_{\theta+\delta}c_{m-\delta}^{(\frac
{m}{2})}=0$ for all $\delta\neq \frac{m}{2}$ and $\frac{m}{2}=k'\leq \theta<m$.

Now assume that $c^{(\frac{m}{2})}_i\neq 0$ for some $0<i\neq \frac{m}{2}<m$, then notice that we may assume that $i<\frac{m}{2}$. Indeed, note that $v_{z^{\frac{m}{2}}}^*=v_{z^{\frac{-m}{2}}}=v_{z^{\frac{m}{2}}}$, thus by comparing the coefficient of $c^{(\frac{m}{2})}_i$ in the Fourier expansion of $E(v_{z^{\frac{m}{2}}})=E(v_{z^{\frac{m}{2}}})^*$, we deduce that $c^{(\frac{m}{2})}_i=\overline{c^{(\frac{m}{2})}_{m-i}}$. Thus we may replace $i$ with $m-i$ if necessary to assume that $0<i<\frac{m}{2}$.

Set $\delta=m-i$ and $\theta=\frac{m}{2}+i$, thus $\theta+\delta=\frac{3m}{2}\equiv \frac{m}{2}~\text{mod}~m$ and $m-\delta=i$, thus  $0=c^{(\frac{m}{2})}_{\frac{m}{2}}c^{(\frac{m}{2})}_i$, hence $c^{(\frac{m}{2})}_{\frac{m}{2}}=0$. Hence $E(v_{z^{\frac{m}{2}}})=0$ as $\|E(v_{z^{\frac{m}{2}}})\|_2^2=c_{\frac{m}{2}}^{(\frac{m}{2})}=0$. This gives us a contradiction.

Thus $c_i^{(\frac{m}{2})}=0$ for all $0<i\neq \frac{m}{2}<m$. This means that $E(v_{z^{\frac{m}{2}}})=c^{(\frac{m}{2})}_{\frac{m}{2}}v_{z^{\frac{m}{2}}}$. Since $E(E(v_{z^{\frac{m}{2}}}))=E(v_{z^{\frac{m}{2}}})$, we deduce that $E(v_{z^{\frac{m}{2}}})=0$ or $v_{z^{\frac{m}{2}}}$. Both lead to a contradiction.

Hence we finish the proof in case $m$ is even.

\textbf{Case 2.} $m$ is odd.

The proof is essentially a simplified version of the above proof of Case 1. Indeed, notice that in this case $z^{k'}\neq -1$ for all $0<k'<m$. Assume the claim fails, then the same argument as in Subcase 1 of Case 1 yields a contradiction and Subcase 2 does not appear.

This finishes the proof of the Claim. \qed

Now, we have proved that $E(v_{z^k})=0$ holds for all $0<k<m$.
    
    Then $E(u_{q-qz^k}v_{z^k})=u_qE(v_{z^k})u_q^*=0$ for all $q\in \mathbf{k}$ and $0<k<m$. This shows that $A=\{E(u_q): q\in \mathbf{k}\}''\subseteq L(\mathbf{k})$ since we already know that $E(u_q)\in L(\mathbf{k})$ for all $q\in \mathbf{k}$. Hence $A$ is abelian and $A=Z(A)=L(\mathbf{k})^F$.
\end{proof}

\begin{remark}\label{remark: direct integral gives more intermediate ones}
    It is known that $L(\mathbf{k})\rtimes F\cong L(\mathbf{k})^F\bar{\otimes}M_{|F|}(\mathbb{C})$, say by an argument using \cite[Theorem 3.29]{Gla}. Thus there do exist abundance of proper intermediate von Neumann subalgebras $A$ with $L(\mathbf
    k)^F\subsetneq A\subsetneq L(\mathbf{k})\rtimes F$ by  the direct integral construction for von Neumann algebras. Hence the $\mathbf{k}\rtimes \mathbf{k}^*$-invariance assumption, which plays the role of ``rigidifying Galois group" that kills off the direct integral intermediate algebras, is necessary and crucial for the classification in the above lemma. 
    We also point out that the inclusion $L(\mathbf{k})^F\subseteq L(\mathbf{k})\rtimes F$ considered here is not covered by all previous known Galois's type theorems as in \cites{BBH, CD, Cho, GK, HSX, ILP, Suz2020}. It is more close to the Bisch-Haagerup type inclusion as considered by Suzuki in his quite recent work \cite[Theorem D]{Suz2026}. However, since $L(\mathbf{k})$ is not a simple C$^*$-algebra, this theorem is still not applicable in our case. Hence, it is reasonable to think of this as a emerging new direction in the study of intermediate subalgebras, following \cites{AJ_IFT, AJZ} in spirit, where non-commutative intermediate factor theorems are proved in both C$^*$ and von Neumann algebra setting.
    \end{remark}

Now we can prove the main result of this paper.

\begin{theorem}\label{thm: invariant vN subalgebras inside L(k rtimes k*)}
Let $\mathbf{k}$ be a countable discrete field of characteristic zero. Then the $G$-invariant von Neumann subalgebras in $L(G)$ for $G=\mathbf{k}\rtimes \mathbf{k}^*$are listed below:

$\mathbb{C}$, $L(\mathbf{k})^{F}$, $L(\mathbf{k}\rtimes H)$, where $F\subseteq \mathbf{k}^*$ is any finite subgroup and $H\subseteq \mathbf{k}^*$ is any  subgroup.
\end{theorem}

\begin{proof}
Let $A\subseteq L(G)$ be a $G$-invariant von Neumann subalgebra. 

Let $\tau$ be the canonical trace on $L(G)$. Denote by $\phi(g)=\tau(g^{-1}E(g))$ and $\psi(g)=\tau(g^{-1}E'(g))$ for all $g\in G$, where $E: L(G)\twoheadrightarrow A$ and $E': L(G)\twoheadrightarrow A'\cap L(G)$ denote the trace preserving conditional expectations.
Note that both $\phi$ and $\psi$ are characters by \cite[Proposition 3.2]{JZ}.

\textbf{Case 1}. $A$ is abelian.

Our goal is to show $A=\mathbb{C}$, $L(\mathbf{k})$ or $L(\mathbf{k})^{F}$ for some non-trivial finite subgroup $F\subseteq \mathbf{k}^*$.

By Corollary \ref{cor: classifying all factor maps}, we deduce that $A\cap L(\mathbf{k})=\mathbb{C}$, $L(\mathbf{k})$ or $L(\mathbf{k})^{F}$ for some non-trivial finite subgroup $F\subseteq \mathbf{k}^*$. 

\textbf{Subcase 1}. $A\cap L(\mathbf{k})=\mathbb{C}$.

Note that  $u_qE(u_p)u_{q}^*=E(u_p)$ for all $p,q\in\mathbf{k}$, which implies that $E(u_p)\in L(\mathbf{k})'\cap L(G)=L(\mathbf{k})$ as $L(\mathbf{k})$ is a masa in $L(G)$ by Proposition \ref{prop: masas}. Thus $E(u_p)\in L(\mathbf{k})\cap A=\mathbb{C}$. By taking trace on both sides, we deduce that $E(u_p)=0$ for all $0\neq p\in\mathbf{k}$. Hence $\phi(p)=0$ for all $0\neq p\in\mathbf{k}$. Note that here $p=(p,1)\in G$.

We show that $A=\mathbb{C}$ in this case.

Indeed, take any $g=(p,r)\in G$ with $r\neq 1$. Note that $\mathbb{Q} \subseteq \mathbf{k}$ since the character of $\mathbf{k}$ is zero. For each $n\in\mathbb{Z}\subseteq \mathbb{Q} \subseteq \mathbf{k}$, write $g_n=(n,1)g(n,1)^{-1}=(n+p-rn,r)\in G$. Then $g_n^{-1}g_m=(r^{-1}(1-r)(m-n),1)\in \mathbf{k}\setminus \{0\}$ if and only if $m\neq n$. Therefore, $\phi(g_n^{-1}g_m)=0$ for all $n\neq m$ by the argument above. Hence $\phi(g)=0$ by Proposition \ref{prop: vanishing characters}. This shows that $\phi(g)=0$ for all $g\neq e\in G$. Equivalently, $E(g)=0$ and hence $A=\mathbb{C}$.

\textbf{Subcase 2}. $A\cap L(\mathbf{k})=L(\mathbf{k})$; equivalently, $L(\mathbf{k})\subseteq A$.

Since $A$ is abelian and $L(\mathbf{k})$ is a masa in $L(G)$ by Proposition \ref{prop: masas}, we deduce that
 $A=L(\mathbf{k})$.

\textbf{Subcase 3}. $A\cap L(\mathbf{k})=L(\mathbf{k})^{F}$ for some non-trivial finite subgroup $F\subseteq \mathbf{k}^*$. In particular, $L(\mathbf{k})^{F}\subseteq A$. Assume $F=\langle z\rangle$, where $z=e^{2\pi i/m}\in \mathbb{T}$ for some $m>1$.

For any $r,s\in \mathbf{k}^*$, we have that $v_sE(v_r)v_s^*=E(v_r)$, hence $E(v_r)\in L(\mathbf{k}^*)'\cap L(G)=L(\mathbf{k}^*)$ as $L(\mathbf{k}^*)$ is a masa in $L(G)$ by Proposition \ref{prop: masas}. 

Note that $L(\mathbf{k})^{F}$ is generated by $\{{\sum_{f\in F}u_{qf}:q\in \mathbf{k}\}}$. 

For any fixed $r\in\mathbf{k}^*$, write $E(v_r)=\sum_{s\in\mathbf{k}^*}c_sv_s$, where $c_s\in\mathbb{C}$.

Since $A$ is abelian and $L(\mathbf{k})^{F}\subseteq A$, we deduce that for all $q\in\mathbf{k}$ the following holds:
\begin{align*}
    (\sum_{f\in F}u_{qf})E(v_r)=E(v_r)(\sum_{f\in F}u_{qf}).
\end{align*}
A calculation shows that
\begin{align*}
    (\sum_{f\in F}u_{qf})E(v_r)&=\sum_{s\in\mathbf{k}^*}\sum_{f\in F}c_su_{qf}v_s;\\
    E(v_r)(\sum_{f\in F}u_{qf})&=\sum_{s\in\mathbf{k}^*}\sum_{f\in F}c_su_{qsf}v_s.
\end{align*}
Thus, comparing the coefficients of $u_qv_s$ on both sides, we deduce that $c_s=0$ for all $s\notin F$. 

\textbf{Claim.} $E(v_r)=0$ for all $1\neq r\in\mathbf{k}^*$.

\textbf{Proof of the claim.}
    From the above, 
    we may write $E(v_r)=\sum_{i=0}^{m-1}c_iv_{z^i}$. 

    It is clear to see that $\|E(v_r)\|_2^2=\tau(v_r^*E(v_r))=0$ for all $r\notin F$.
    Therefore, $E(v_r)=0$ for all $r\notin F$.

    We are left to show that $E(v_r)=0$ for all $r\in F\setminus\{1\}$.

    For any $0\neq q\in \mathbf{k}$, we have 
    \begin{align*}
        u_qE(v_r)u_q^*=\sum_{i=0}^{m-1}c_iu_qv_{z^i}u_q^*=\sum_{i=0}^{m-1}c_iu_{q-qz^i}v_{z^i}.   
    \end{align*}
    
    Note that both $u_qE(v_r)u_q^*$ and $E(v_r)$ belong to $A$ and hence commute with each other. So let us compute 
\begin{align*}
    u_qE(v_r)u_q^*E(v_r)=(\sum_{i=0}^{m-1}c_iu_{q-qz^i}v_{z^i})(\sum_{j=0}^{m-1}c_jv_{z^j})=\sum_{i,j=0}^{m-1}c_ic_ju_{q-qz^i}v_{z^{i+j}}.
\end{align*}
It is not hard to see that the coefficient of $u_{q-qz^i}v_{z^{i+j}}$ is indeed $c_ic_j$, i.e. there is only one term in the above expression supported at $u_{q-qz^i}v_{z^{i+j}}$. 

Similarly, 
\begin{align*}
    E(v_r)u_qE(v_r)u_q^*=(\sum_{j'=0}^{m-1}c_{j'}v_{z^{j'}})(\sum_{i'=0}^{m-1}c_{i'}u_{q-qz^{i'}}v_{z^{i'}})=\sum_{i',j'=0}^{m-1}c_{i'}c_{j'}u_{z^{j'}(q-qz^{i'})}v_{z^{i'+j'}}.
\end{align*}

If $u_{q-qz^i}v_{z^{i+j}}=u_{z^{j'}(q-qz^{i'})}v_{z^{i'+j'}}$, then we deduce that 
\begin{align*}
    1-z^i&=z^{j'}(1-z^{i'})=z^{j'}-z^{i'+j'}.\\
    i+j&=i'+j'~(mod~m).
\end{align*}

From the first equality, we may apply Lemma \ref{lem: difference between two complex numbers with modulus one} to deduce three cases.

Case 1. $i=0$.

Then $i'=0$.

Case 2. $i\neq 0$ and $z^{j'}=1$ and $z^i=z^{i'+j'}=z^{i+j}$.

This implies that $j'=0$ and $j=0$.

Case 3. $i\neq 0$ and $-z^i=z^{j'}$ and $1=-z^{i'+j'}=-z^{i+j}$.

Hence, if we consider $i\neq 0\neq j$ and $z^{i+j}\neq -1$, then we deduce that $c_ic_j=0$ since no $(i',j')$ exists such that $u_{q-qz^i}v_{z^{i+j}}=u_{z^{j'}(q-qz^{i'})}v_{z^{i'+j'}}$. In particular, for $i=j\neq 0$ and $z^{2i}\neq -1$, we deduce that $c_i=0$.

Hence, for odd $m$, note that $z^{2i}\neq -1$ holds true for all $i\neq 0$. Indeed, if $z^{2i}=-1$, then $z^{4i}=1$ and thus $m\mid 4i$. Since $0\leq i\leq m-1$, this implies that $4i\in \{2m, 3m\}$, which contradicts to $2\nmid m$. 

This shows that for odd $m$, we have $c_i=0$ for all $i\neq 0$. Hence $E(v_r)=c_0$. Taking trace on both sides, we deduce that $c_0=0$.

We are left to consider even $m$. Then note that $z^{2i}=-1$ means $2i=\frac{m}{2}~(mod~m)$. 
Since $0<i\leq m-1$, we deduce that $i=\frac{m}{4}$ or $i=\frac{3m}{4}$.
Hence, $c_i=0$ for all $0<i\leq m-1$ and $i\neq \frac{m}{4}, \frac{3m}{4}$. Also note that if we take $i=\frac{m}{4}$ and $j=\frac{3m}{4}$, then $z^{i+j}=z^{m}=1\neq -1$. Hence none of Case 1, 2 and 3 are valid, therefore $c_ic_j=0$, i.e. either $c_{\frac{m}{4}}=0$ or $c_{\frac{3m}{4}}=0$.

If $c_{\frac{m}{4}}=0$, then $E(v_r)=c_0+c_{\frac{3m}{4}}v_{z^{\frac{3m}{4}}}$.
Take $i=j=\frac{3m}{4}$ as above, we are in Case 3 and hence $c^2_{\frac{3m}{4}}=c_{i'}c_{j'}$, while $-z^{\frac{3m}{4}}=z^{j'}$. Clearly, this condition implies that $j'\not\in\{ 0,\frac{3m}{4}\}$. Hence $c^2_{\frac{3m}{4}}=0$, i.e. $c_{\frac{3m}{4}}=0$, then $E(v_r)=c_0$. Taking trace on both sides, we deduce that $E(v_r)=0$.

If $c_{\frac{3m}{4}}=0$, then $E(v_r)=c_0+c_{\frac{m}{4}}v_{z^{\frac{m}{4}}}$. Take $i=j=\frac{m}{4}$ as above, we are in Case 3 and hence $c^2_{\frac{m}{4}}=c_{i'}c_{j'}$, while $-z^{\frac{m}{4}}=z^{j'}$. This condition implies that $j'\not\in \{0, \frac{m}{4}\}$. Hence $c^2_{\frac{m}{4}}=0$. Arguing similar as above, we deduce that $E(v_r)=0$.

This finishes the proof of the above Claim.\qed

Then for any $g=(p,r)\in G$ with $r\neq 1$, we get that \[E(g)=E(u_pv_r)=E(u_{p/(1-r)}v_ru_{p/(1-r)}^*)=u_{p/(1-r)}E(v_r)u_{p/(1-r)}^*=0.\] Thus $A\subseteq L(\mathbf{k})$. Hence $A=A\cap L(\mathbf{k})=L(\mathbf{k})^{F}$.

\textbf{Case 2}. $A$ is a subfactor.

Following Chifan-Das-Sun's approach, see the proof of \cite[Theorem 3.1]{CDS}, there exists some normal subgroup $N\lhd G$ such that $L(N)=A\bar{\otimes}(A'\cap L(N))$. Then from the calculation used in the proof of \cite[Theorem 3.3]{DJ}, we know that $\phi(g)\psi(g)=0$ for all $e\neq g\in N$. 

Observe that $N\cap \mathbf{k}\neq \{e\}$. Indeed, otherwise, since $N$ and $\mathbf{k}$ are normal in $G$, we have that $N$ commutes with $\mathbf{k}$, which implies $N\subseteq C_G(\mathbf{k})=\mathbf{k}$. Hence $A\subseteq L(N)\subseteq L(\mathbf{k})$ which is abelian, contradicting to our assumption that $A$ is a subfactor. 

Next, observe that since $\mathbf{k}^*\curvearrowright\mathbf{k}^*$ is transitive and the normal subgroup $N\cap \mathbf{k}\neq \{e\}$, we deduce that $N\cap \mathbf{k}=\mathbf{k}$, i.e. $\mathbf{k}\subseteq N$. 

Then, note that any $G$-invariant character on $\mathbf{k}$ can be written as $c\delta_{e}+(1-c)1_{\mathbf{k}}$ by  \cite[Lemma 12.B.2]{BD}. So from $\phi(q)\psi(q)=0$ for all $0\neq q\in\mathbf{k}$, we deduce that $\phi|_{\mathbf{k}}=\delta_{e}$ or $\psi|_{\mathbf{k}}=\delta_{e}$.

\textbf{Subcase 1}. $\phi|_{\mathbf{k}}=\delta_e$.

Then $E(u_q)=0$ for all $0\neq q\in\mathbf{k}$.

By using Proposition \ref{prop: vanishing characters} and arguing similarly as in the proof of Subcase 1 in Case 1, it is not hard to show that $\phi((q,r))=0$ for all $q\in\mathbf{k}$ and $r\neq 1$. Thus $A=\mathbb{C}$.

\textbf{Subcase 2}. $\psi|_{\mathbf{k}}=\delta_e$.

Similar to the above subcase 1, we get that $A'\cap L(N)=\mathbb{C}$ and thus $L(N)=A\bar{\otimes}\mathbb{C}=A$, so $A=L(N)$ for some normal subgroup $\mathbf{k}\subseteq N\subseteq G$. Since $A$ is a subfactor, $N$ is an i.c.c. group. Then it is clear that $N=\mathbf{k}\rtimes H$ for some infinite subgroup $H\subseteq \mathbf{k}^*$.

\textbf{Case 3}. General case.

Note that the center $Z(A)$ is still $G$-invariant and abelian. By Case 1, we may further consider three subcases.

\textbf{Subcase 1}. $Z(A)=\mathbb{C}$; equivalently, $A$ is a $G$-invariant subfactor. Then the proof of Case 2 shows that $A=\mathbb{C}$ or $A=L(N)$ for some normal subgorup $\mathbf{k}\subseteq N\subseteq G$.

\textbf{Subcase 2}. $Z(A)=L(\mathbf{k})$.

Then $A\subseteq Z(A)'\cap L(G)=L(\mathbf{k})$, hence $A$ is abelian. Thus $A=Z(A)=L(\mathbf{k})$.

\textbf{Subcase 3}.
$Z(A)=L(\mathbf{k})^{F}$ for some non-trivial finite subgroup $F\subseteq \mathbf{k}^*$.

By Lemma \ref{lem: invariant A with Z(A)=L(k)^{F}}, we deduce that $A=L(\mathbf{k})^{F}$ or $A=L(\mathbf{k}\rtimes{F})$, which completes the whole proof. 
\end{proof}

\begin{corollary}\label{cor: invariant vN subalgebras in L(G×H)}
    Let $G=\mathbf{k}\rtimes\mathbf{k}^*$ as in Theorem \ref{thm: invariant vN subalgebras inside L(k rtimes k*)}. Let $H$ be a countably infinite non-amenable group with exactly 2 conjugacy classes (as in \cite{Osin}). Then the $K$-invariant von Neumann subalgebras in $L(K)$ for $K=G\times H$ are listed below: 

    $L(\mathbf{k})^{F}$, $L(\mathbf{k})^{F} \bar{\otimes} L(H)$, $L(N)$, $L(N\times H)$, where $F\subseteq \mathbf{k}^*$ is any finite subgroup and $N\lhd G$ is any  normal subgroup.
\end{corollary}
\begin{proof}
    Let $A\subseteq L(K)$ be a $K$-invariant von Neumann subalgebra. 
    
    Let $\tau$ be the canonical trace on $L(K)$. Denote by $\phi(s)=\tau(s^{-1}E(s))$ for all $s\in K$, where $E: L(K)\twoheadrightarrow A$ denotes the trace preserving conditional expectation.

    Note that $H$ must be i.c.c. and simple. Moreover, as explained in the proof of \cite[Theorem B]{DJ}, $H$ has the ISR property. Hence the only $H$-invariant von Neumann subalgebras in $L(H)$ are either $\mathbb{C}$ or $L(H)$.

    Since $A\cap L(H)$ is a $H$-invariant von Neumann subalgebra, we may consider two cases.

    \textbf{Case 1}. $A\cap L(H)=\mathbb{C}$.
    
    Note that $E(h)\in L(G)'\cap L(K)\subseteq L(H)$ for all $h\in H$. Thus $E(h)\in A\cap L(H)=\mathbb{C}$. By taking trace on both sides, we deduce that $E(h)=0$ for all $h\neq e$.  Thus we have that $\phi|_{H}=\delta_e$ by the definition of $\phi$. 

    \textbf{Claim.} $A\subseteq L(G)$.

    \textbf{Proof of the Claim.}
        We only need to show that $\phi(g,h)=0$ for all $g\in G$ and $e\neq h\in H$. 
Indeed, once this holds, then $\|E((g,h))\|_2^2=\phi(g,h)=0$ and thus $E(g,h)=0$ for all $g\in G$ and $e\neq h\in H$. Hence for any $x=\sum_{(g,h)\in K}c_{(g,h)}(g,h)\in A$, if $(g,h)\in K \setminus G$, i.e. $h\neq e$, then $c_{(g,h)}=\tau((g,h)^{-1}x)=\tau(E((g,h)^{-1}x))=\tau(E((g,h)^{-1})x)=\tau(0·x)=0$. Thus $x\in L(G)$.
        
        Take any $(g,h)\in K$ with $h\neq e$.
        Since $H$ is i.c.c. and $h\neq e$, there exists $(h_n)_{n\in \mathbb{Z}}$ in $H$ such that $h_nhh_n^{-1} \neq h_mhh_m^{-1}$ if $m\neq n$.  
        For each $n\in \mathbb{Z}$, write $s_n=(e,h_n)(g,h)(e,h_n)^{-1}=(g,h_nhh_n^{-1})\in K$. It is easy to check that $\phi(s_n^{-1}s_m)=0$ if $m\neq n$. Hence $\phi(g,h)=0$ by Proposition \ref{prop: vanishing characters}.  
    \end{proof}

    Then by Theorem \ref{thm: invariant vN subalgebras inside L(k rtimes k*)}, $A=L(\mathbf{k})^{F}$ for some finite subgroup $F\subseteq \mathbf{k}^*$ or $A=L(N)$ for some normal subgroup $N\lhd G$. This finishes the proof of Case 1.
    
    \textbf{Case 2}. $A\cap L(H)=L(H)$.
    Then $L(H)\subseteq A\subseteq L(K)=L(G)\bar{\otimes}L(H)$. 
    
    By Ge-Kadison's splitting theorem \cite{GK}, $A=Q\bar{\otimes}L(H)$ for some von Neumann subalgebra $Q\subseteq L(G)$. 
    Note that $Q$ is $G$-invariant as $A$ is. Thus Theorem \ref{thm: invariant vN subalgebras inside L(k rtimes k*)} yields that $Q=L(\mathbf{k})^{F}$ for some finite subgroup $F\subseteq \mathbf{k}^*$ or $Q=L(N)$ for some normal subgroup $N\lhd G$. This finishes the proof of Case 2.
\qed

\begin{remark}\label{remark: construct 2^n many exotic invariant subalgebras}
Let $H$ be the group considered as in the above corollary.
    For any $n\geq 1$, set $K_n=G\oplus \oplus_{i=1}^nH$, where $G$ is the i.c.c. group constructed in \cite[Theorem B]{ADJS} such that $L(G)$ admits only one $G$-invariant von Neumann subalgebras not arising from subgroups of $G$. Then the same proof as above shows that the number of $K_n$-invariant von Neumann subalgebras in $L(K_n)$ not arising from subgroups of $K_n$ is $2^n$.
\end{remark}

Taking $\mathbf{k}=\mathbb{Q}$ in Theorem \ref{thm: invariant vN subalgebras inside L(k rtimes k*)}, we classify all invariant von Neumann subalgebras of $L(\mathbb{Q}\rtimes \mathbb{Q}^{\times})$.

\begin{corollary}
The $G$-invariant von Neumann subalgebras in $L(G)$ for $G=\mathbb{Q}\rtimes \mathbb{Q}^{\times}$ are listed below:

$\mathbb{C}$, $L(\mathbb{Q})^{\{\pm 1\}}$, $L(\mathbb{Q})$, $L(\mathbb{Q}\rtimes H)$, where $H\subseteq \mathbb{Q}^{\times}$ is any non-trivial subgroup.
\end{corollary}

\begin{remark}
    The above corollary provides yet another example of $G$ such that there is precisely only one exotic invariant subalgebras in $L(G)$. Compared with the first such an example as constructed in \cite[Theorem B]{ADJS}, which relies  on    combinatorial arguments,
   our proof has more ergodic theoretic flavor. In fact, if we directly set $\mathbf{k}=\mathbb{Q}$ in the whole process of proof,  then most steps, e.g. Lemma \ref{lem: invariant A with Z(A)=L(k)^{F}} can actually be greatly simplified. We left the details for the readers to check. 
\end{remark}

The following proves the first part of Theorem \ref{thmA}.
\begin{corollary}\label{cor: group with n-many exotic invariant vN subalgebras inside L(G)}
Let $n\geq 1$. Then there exists an amenable i.c.c. group $G_{n}$ such that there are precisely $n$-many $G_{n}$-invariant von Neumann subalgebras inside $L(G_{n})$ not arising from subgroups of $G_n$.
\end{corollary}
\begin{proof}
    By Lemma \ref{lem: k* which has exactly n-many subgroups}, there exists a countable field $\mathbf{k}_n$ of characteristic zero such that there are precisely $n$-many non-trivial finite subgroups inside $\mathbf{k}_n^*$.  
    To finish the proof, just take $G_n=\mathbf{k}_n\rtimes\mathbf{k}_n^*$ and apply Theorem \ref{thm: invariant vN subalgebras inside L(k rtimes k*)}. Indeed, it is easy to check that $G_n$ is amenable and i.c.c. Besides, by considering the support of Fourier expansion of elements in $L(\mathbf{k}_n)^F$, where $F\subset \mathbf{k}_n^*$ is any non-trivial finite subgroup, it is clear that $L(\mathbf{k}_n)^F\neq L(H)$ for any subgroup $H\subseteq G_n$. In other words, $L(\mathbf{k}_n)^F$ does not arise from subgroups of $G_n$.
\end{proof}

\begin{remark}
    Let $\mathbf{k}$ be a countable discrete field of characteristic zero such that there are infinitely many finite subgroups inside $\mathbf{k}^*$ (such $\mathbf{k}$ exists since we can take $\mathbf{k}$ to be the $\mathbb{Q}$-algebraic extension generated by $\cup_n\mathbb{Q}(\zeta_n)$). Set $G=\mathbf{k}\rtimes\mathbf{k}^*$, then Theorem \ref{thm: invariant vN subalgebras inside L(k rtimes k*)} yields that there are infinitely many $G$-invariant von Neumann subalgebras inside $L(G)$ not arising from subgroups of $G$.
\end{remark}

The following proves the second part of Theorem \ref{thmA}.

\begin{corollary}\label{cor: non-amenable group with n-many exotic invariant vN subalgebras inside L(G)}
Let $n\geq 1$. Then there exists a non-amenable i.c.c. group $H_{n}$ such that there are precisely $2n$-many $H_{n}$-invariant von Neumann subalgebras inside $L(H_{n})$ not arising from subgroups of $H_n$.
\end{corollary}
\begin{proof}
    Let $G_n$ be the group as in Corollary \ref{cor: group with n-many exotic invariant vN subalgebras inside L(G)}.
    Let $H$ be a countably infinite non-amenable group with exactly 2 conjugacy classes. 
    Set $H_n:=G_n\times H$ and apply Corollary \ref{cor: invariant vN subalgebras in L(G×H)}, then we obtain the desired result. 
\end{proof}


\begin{bibdiv}
\begin{biblist}

\bib{AB}{article}{
   author={Alekseev, V.},
   author={Brugger, R.},
   title={A rigidity result for normalized subfactors},
   journal={J. Operator Theory},
   volume={86},
   date={2021},
   number={1},
   pages={3--15},}

\bib{A-relative}{article}{
author={Amrutam, T.},
title={On relative invariant subalgebra rigidity property
},
year={2026},
status={arXiv: 2604.04835},
}

\bib{ADJS}{article}{
author={Amrutam, T.},
author={Dudko, A.},
author={Jiang, Y.},
author={Skalski, A.},
title={Invariant subalgebras rigidity for von Neumann algebras of groups arising as certain semidirect products},
year={2025},
status={arXiv: 2507.12824},
}


\bib{AHO}{article}{
author={Amrutam, T.},
author={Hartman, Y.},
author={Oppelmayer, H.},
title={On the amenable subalgebras of group von Neumann algebras},
 journal={J. Funct. Anal.},
   volume={288},
   date={2025},
   number={2},
   pages={Paper No. 110718, 20 pp},
}

\bib{AJ}{article}{
   author={Amrutam, T.},
   author={Jiang, Y.},
   title={On invariant von Neumann subalgebras rigidity property},
   journal={J. Funct. Anal.},
   volume={284},
   date={2023},
   number={5},
   pages={Paper No. 109804, 26 pp.},}

   \bib{AJ_IFT}{article}{
   author={Amrutam, T.},
   author={Jiang, Y.},
   title={Splitting of tensor products and intermediate factor theorem:
   continuous version},
   journal={J. Lond. Math. Soc. (2)},
   volume={111},
   date={2025},
   number={6},
   pages={Paper No. e70205, 35},}


   \bib{AJZ}{article}{
author={Amrutam, T.},
author={Jiang, Y.},
author={Zhou, S.},
title={Non-commutative Intermediate Factor theorem associated with W$^*$-dynamics of product groups}
date={2025},
note={arXiv: 2508.18978},
}

\bib{BBH}{article}{
   author={Bader, U.},
   author={Boutonnet, R.},
   author={Houdayer, C.},
   title={Charmenability of higher rank arithmetic groups},
   language={English, with English and French summaries},
   journal={Ann. H. Lebesgue},
   volume={6},
   date={2023},
   pages={297--330},}

  \bib{Bek}{article}{
   author={Bekka, B.},
   title={Operator-algebraic superridigity for ${\rm SL}_n(\Bbb Z)$, $n\geq
   3$},
   journal={Invent. Math.},
   volume={169},
   date={2007},
   number={2},
   pages={401--425},}

\bib{BD}{book}{
   author={Bekka, B.},
   author={de la Harpe, P.},
   title={Unitary representations of groups, duals, and characters},
   series={Mathematical Surveys and Monographs},
   volume={250},
   publisher={American Mathematical Society, Providence, RI},
   date={2020},
   pages={xi+474},}
   
\bib{BF}{article}{
author={ Bj{\"o}rklund, M.}, 
author={ Fish, A.}, 
title={Dynamics of multiplicative groups over fields and F{\o}lner-Kloosterman sums},
status={arXiv: 2512.07106},
date={2025},}

\bib{CD}{article}{
   author={Chifan, I.},
   author={Das, S.},
   title={Rigidity results for von Neumann algebras arising from mixing
   extensions of profinite actions of groups on probability spaces},
   journal={Math. Ann.},
   volume={378},
   date={2020},
   number={3-4},
   pages={907--950},}

\bib{CDS}{article}{
   author={Chifan, I.},
   author={Das, S.},
   author={Sun, B.},
   title={Invariant subalgebras of von Neumann algebras arising from
   negatively curved groups},
   journal={J. Funct. Anal.},
   volume={285},
   date={2023},
   number={9},
   pages={Paper No. 110098, 28 pp.},}


\bib{Cho}{article}{
   author={Choda, H.},
   title={A Galois correspondence in a von Neumann algebra},
   journal={Tohoku Math. J. (2)},
   volume={30},
   date={1978},
   number={4},
   pages={491--504},}
   
\bib{DJ}{article}{
  author={Dudko, A.},
  author={Jiang, Y.},
  title={A character approach to the ISR property},
  status={arXiv: 2410.14517v3},
  year={2024}}
  

\bib{GK}{article}{
   author={Ge, L.},
   author={Kadison, R.},
   title={On tensor products for von Neumann algebras},
   journal={Invent. Math.},
   volume={123},
   date={1996},
   number={3},
   pages={453--466},}


\bib{DM}{article}{
   author={Dudko, A.},
   author={Medynets, K.},
   title={Finite factor representations of Higman-Thompson groups},
   journal={Groups Geom. Dyn.},
   volume={8},
   date={2014},
   number={2},
   pages={375--389},}

  

\bib{Gla}{book}{
author={Glasner, E.},
title={Ergodic Theory via Joinings},
series={Mathematical Surveys and Monographs},
volume={101},
publisher={American Mathematical Society},
date={[2003] \copyright 2003},
pages={384},}

\bib{HSX}{article}{
   author={Hu, X.},
   author={Shi, R.},
   author={Xu, F.},
   title={On a composition of subfactors with group subfactors},
   journal={Sci. China Math.},
   volume={64},
   date={2021},
   number={2},
   pages={373--384},}
   
\bib{ILP}{article}{
   author={Izumi, M.},
   author={Longo, R.},
   author={Popa, S.},
   title={A Galois correspondence for compact groups of automorphisms of von
   Neumann algebras with a generalization to Kac algebras},
   journal={J. Funct. Anal.},
   volume={155},
   date={1998},
   number={1},
   pages={25--63},}
  
\bib{jiangli}{article}{
author={Jiang, Y.},
author={Li, H.},
title={Classification of Invariant Subalgebras in a class of factors with property (T)},
status={arXiv: 2601.06353},
year={2026},
}

\bib{jiangliu}{article}{
   author={Jiang, Y.},
   author={Liu, R.},
   title={On invariant subalgebras when the ISR property fails},
   journal={J. Operator Theory},
   volume={95},
   date={2026},
   number={1},}


\bib{JZ}{article}{
author={Jiang, Y.},
author={Zhou, X.},
title={An example of an infinite amenable group with the ISR property},
journal={Math. Z.},
   volume={307},
   date={2024},
   number={2},
   pages={Paper No. 23, 11 pp},
}


\bib{KP}{article}{
   author={Kalantar, M.},
   author={Panagopoulos, N.},
   title={On invariant subalgebras of group and von Neumann algebras},
   journal={Ergodic Theory Dynam. Systems},
   volume={43},
   date={2023},
   number={10},
   pages={3341--3353},}

\bib{KL_book}{book}{
   author={Kerr, D.},
   author={Li, H.},
   title={Ergodic theory Independence and dichotomies},
   series={Springer Monographs in Mathematics},
   publisher={Springer, Cham},
   date={2016},
   pages={xxxiv+431},}

\bib{KV}{article}{
   author={Krogager, A.},
   author={Vaes, S.},
   title={A class of ${\rm II}_1$ factors with exactly two group measure
   space decompositions},
   language={English, with English and French summaries},
   journal={J. Math. Pures Appl. (9)},
   volume={108},
   date={2017},
   number={1},
   pages={88--110},}

\bib{Lin}{article}{
   author={Linnell, P. A.},
   title={Zero divisors and group von Neumann algebras},
   journal={Pacific J. Math.},
   volume={149},
   date={1991},
   number={2},
   pages={349--363},}

\bib{moor}{article}{
author={Moorhouse, G. E.},
title={Cyclictomic fields with applications},
year={2018},
note={lecture note available at \url{https://ericmoorhouse.org/handouts/cyclotomic_fields.pdf}},
}

   \bib{Mor}{book}{
   author={Morandi, P.}, 
   title={Field and Galois Theory}, 
   series={Graduate Texts in Mathematics},
   volume={167},
   publisher={Springer-Verlag New York, Inc},
   date={[1996] \copyright 1996},
   pages={x+281,}}
   

\bib{Osin}{article}{
   author={Osin, D.},
   title={Small cancellations over relatively hyperbolic groups and
   embedding theorems},
   journal={Ann. of Math. (2)},
   volume={172},
   date={2010},
   number={1},
   pages={1--39},}



\bib{Schmidt_book}{book}{
   author={Schmidt, K.},
   title={Dynamical systems of algebraic origin},
   series={Progress in Mathematics},
   volume={128},
   publisher={Birkh\"{a}user Verlag, Basel},
   date={1995},
   pages={xviii+310},}

\bib{masa_book}{book}{
   author={Sinclair, A.},
   author={Smith, R.},
   title={Finite von Neumann algebras and masas},
   series={London Mathematical Society Lecture Note Series},
   volume={351},
   publisher={Cambridge University Press, Cambridge},
   date={2008},
   pages={x+400},}

  \bib{Suz2020}{article}{
   author={Suzuki, Y.},
   title={Complete descriptions of intermediate operator algebras by
   intermediate extensions of dynamical systems},
   journal={Comm. Math. Phys.},
   volume={375},
   date={2020},
   number={2},
   pages={1273--1297},}

\bib{Suz2026}{article}{
   author={Suzuki, Y.},
   title={Crossed product splitting of intermediate operator algebras via
   2-cocycles},
   journal={Math. Ann.},
   volume={394},
   date={2026},
   number={2},
   pages={Paper No. 38, 37},}
   
\end{biblist}
\end{bibdiv}
\end{document}